\documentclass[reqno,a4paper]{article}
\usepackage{a4wide}
\usepackage[centertags]{amsmath}
\usepackage{amsfonts, amssymb, amsthm, dsfont, mathrsfs, wasysym, esint}
\usepackage[draft]{optional}
\usepackage[usenames]{color}
\definecolor{citegreen}{rgb}{0,0.6,0}
\definecolor{refred}{rgb}{0.8,0,0}
\usepackage[colorlinks, citecolor=citegreen, linkcolor=refred]{hyperref}

\title{Qualitative and quantitative estimates for minimal hypersurfaces with bounded index and area}
\author{Reto Buzano and Ben Sharp}
\date{}

\theoremstyle{plain}
\newtheorem{lemma}{Lemma}[section]
\newtheorem{theorem}[lemma]{Theorem}
\newtheorem{prop}[lemma]{Proposition}
\newtheorem{cor}[lemma]{Corollary}

\theoremstyle{definition}
\newtheorem{defn}[lemma]{Definition}

\newtheorem{remark}[lemma]{Remark}  
   
\newcommand{\pl}[2]{{\frac{\partial #1}{\partial #2}}}
\newcommand{\ppl}[3]{{\frac{\partial^2 #1}{\partial #2 \partial #3}}}

\newcommand{\R}{\mathbb{R}}
\newcommand{\N}{\mathbb{N}}
\newcommand{\h}{\mathcal{H}}
\newcommand{\A}{\mathcal{A}}
\newcommand{\ed}{{\rm d}}
\newcommand{\id}{\;\ed}
\newcommand{\D}{{\nabla}}
\newcommand{\eps}{{\varepsilon}}
\newcommand{\emb}{\hookrightarrow}
\newcommand{\cemb}{\subset \subset}

\newcommand{\twopartdef}[4]
{ \left\{\begin{array}{ll}
			#1 & \mbox{if  } #2 \bigskip \\
			#3 & \mbox{if  } #4
\end{array}\right.}
\newcommand{\fourpartdef}[8]
{\left\{\begin{array}{llll}
			#1 & \mbox{if  } #2\bigskip \\
			#3 & \mbox{if  } #4 \bigskip \\
			#5 & \mbox{if  } #6 \bigskip \\
			#7 & \mbox{if  } #8
\end{array}\right.}

\makeatletter
\def\blfootnote{\gdef\@thefnmark{}\@footnotetext}
\makeatother
\begin{document}
\maketitle
\blfootnote{\textup{2010} \textit{Mathematics Subject Classification}: 53A10 (primary), 49Q05, 58E12 (secondary)}
\begin{abstract}
We prove qualitative estimates on the total curvature of closed minimal hypersurfaces in closed Riemannian manifolds in terms of their index and area, restricting to the case where the hypersurface has dimension less than seven. In particular, we prove that if we are given a sequence of closed minimal hypersurfaces of bounded area and index, the total curvature along the sequence is quantised in terms of the total curvature of some limit hypersurface, plus a sum of total curvatures of complete properly embedded minimal hypersurfaces in Euclidean space -- all of which are finite. Thus, we obtain qualitative control on the topology of minimal hypersurfaces in terms of index and area as a corollary.    
\end{abstract}

%------------------------------------------------------------------
\section{Introduction and main results}

In this article, we investigate the class $\mathfrak{M}(N)$ of closed, smooth, embedded minimal hypersurfaces $M^n$ in a closed Riemannian manifold $(N^{n+1},g)$. By the work of Almgren-Pitts \cite{A68, P81} and Schoen-Simon \cite{ScS81}, we know that $\mathfrak{M}(N)$ is non-empty when $2 \leq n\leq 6$. In fact, if the ambient manifold $N$ has positive Ricci curvature, $Ric_N > 0$, then $\mathfrak{M}(N)$ contains at least countably many elements due to a recent result of Marques-Neves \cite{MN13} using min-max techniques. Moreover, by construction, their minimal hypersurfaces have bounded Morse index and bounded area and it thus seems natural to study the properties of subsets of $\mathfrak{M}(N)$ satisfying these two conditions.

There are many known relationships between the index of minimal submanifolds and their topological and analytic properties. In particular, given an immersed minimal submanifold in a closed Riemannian manifold, one knows that the Morse index is bounded linearly from above in terms of area and total curvature, see Ejiri-Micallef \cite{EM08} for $n=2$ and Cheng-Tysk \cite{CT94} for $n\geq 3$. Moreover, for properly immersed minimal hypersurfaces in Euclidean space, we have the same result -- except this time the upper bound is purely in terms of the total curvature (and thus implicitly contains information on the number of ends the minimal hypersurface has). 

Unlike the upper bound case, the lower bound case is heavily dependant on codimension: One would never expect index and volume growth to control the topology in the case of complex projective sub-varieties of $\mathbb{C}^2$ (which are calibrated and therefore local minima for area). However, for hypersurfaces, this seems more conceivable, and indeed there are several results in the case where the ambient manifold is Euclidean space. In particular, one knows that an embedded minimal hypersurface in Euclidean space with finite index and Euclidean volume growth has finite total curvature by results of Fischer-Colbrie \cite{FC85} for $n=2$ (see also Chodosh-Maximo \cite{CM14}) and Tysk \cite{T89} for $n\geq 3$ -- though in two dimensions one does not need to assume Euclidean volume growth. Moreover, Li-Wang \cite{LW02} proved that a minimal hypersurface in Euclidean space with finite index has finite first Betti number. For minimal hypersurfaces in closed manifolds, similar results are currently only known in a few special cases, for example Urbano \cite{U90} proved that when $N$ is a three-sphere and $M$ is compact, orientable, and not totally geodesic, then the Morse index of $M$ is at least $5$. However, in the cases where lower bounds for the index are known, they can be extremely useful; the above mentioned result by Urbano for example plays a crucial role in the resolution of the Willmore conjecture by Marques-Neves \cite{MN14}.

The purpose of this article is to give \emph{qualitative lower bounds} on index and area in terms of total curvature for embedded minimal hypersurfaces in a closed Riemannian manifold $N$. As a corollary, we then obtain that the index and area give a qualitative upper bound on the topology of embedded minimal hypersurfaces. This last result (see Corollary \ref{cor:Betti} below) has recently also been proved by Chodosh-Ketover-Maximo \cite{CKM15} when $n=2$, using different techniques (which yield finer results in two dimensions). Furthermore the methods in \cite{CKM15} also yield a topological bound when $3\leq n \leq 6$ -- however they do not obtain any analytical control for example on the total curvature.   

In order to state our main results, let us first set up the necessary notation. For an embedded and closed minimal hypersurface $M^n\emb (N^{n+1},g)$ in a closed Riemannian manifold $N$ with $n\leq 6$, we define the total curvature by
\begin{equation}
\A(M) := \int_M |A|^n \id V_M,
\end{equation}
where $A$ denotes the second fundamental form of the embedding. We define the Morse index of $M$, denoted $index(M)$, as the number of negative eigenvalues associated with the Jacobi (second variation) operator for minimal hypersurfaces $M\subset N$: 
\begin{equation}
Q(v,v):=\int_{M}\Big( |\D^\bot v|^2 - |A|^2|v|^2 -Ric_N(v,v) \Big) \id V_M,
\end{equation} 
where $v\in\Gamma({\rm{N}}M)$ is a section of the normal bundle, $\D^\bot$ is the normal connection, and of course $Ric_N$ is the Ricci curvature of $N$. Moreover, we denote the $p^{th}$ eigenvalue of the Jacobi operator by $\lambda_p$. We then define 
\begin{equation}
\mathcal{M}(\Lambda,I) := \{ M \in \mathfrak{M}(N) \mid \h^n(M) \leq \Lambda,\ index(M) \leq I\},
\end{equation}
the set of closed, smooth and embedded minimal hypersurfaces with area and Morse index bounded from above (where by area we mean $\h^n(M)$, the $n$-dimensional Hausdorff measure of $M$). 

Similarly,
\begin{equation}
\mathcal{M}_p(\Lambda,\mu) := \{ M \in \mathfrak{M}(N) \mid \h^n(M) \leq \Lambda,\ \lambda_p \geq -\mu\}
\end{equation}
is the set of elements in $\mathfrak{M}(N)$ with bounded area and $p^{th}$ eigenvalue $\lambda_p$ of the Jacobi operator bounded from below.

In \cite{Sha15}, the second author proves a compactness theorem for the set $\mathcal{M}(\Lambda,I)$, and in later joint work Ambozio-Carlotto-Sharp \cite{ACS15} also prove a compactness theorem for $\mathcal{M}_p(\Lambda,\mu)$. In particular, we know that given a sequence of minimal hypersurfaces in $\mathcal{M}(\Lambda,I)$ (or $\mathcal{M}_p(\Lambda,\mu)$), there is some smooth  limit in the same class and our sequence sub-converges smoothly and graphically (with finite multiplicity $m$) away from a discrete set $\mathcal{Y}$ on the limit. In the present paper, we study more carefully what is happening at this discrete set. Using a localised version of the aforementioned compactness result, Theorem \ref{thm:oldmain} below, it is possible to re-scale at appropriately chosen levels so that the re-scaled hypersurfaces converge to some limit in Euclidean space. However, this limit will have finite index and Euclidean volume growth, at which point the result of Tysk \cite{T89} gives us a total curvature bound. We will show that we can perform a bubbling argument that captures all of the total curvature along the sequence. Of course, in order to do so, we have to worry about the intermediate (or so called ``neck'') domains. By construction, these neck domains will be graphical with bounded slope and small dyadic total curvatures, at which point we carry out a detailed analysis of their structure and check that we cannot lose any total curvature in these regions -- see section \ref{NA}. Thus, all of the total curvature is captured by the ``bubbles'' and the limit only, and we obtain the following main result.

\begin{theorem}\label{thm:main}
Let $2\leq n \leq 6$ and $N^{n+1}$ be a smooth closed Riemannian manifold. If $\{M_k^n\}\subset \mathcal{M}(\Lambda,I)$ for some fixed constants $\Lambda\in \R$, $I\in \N$ independent of $k$, then up to subsequence, there exists $M\in \mathcal{M}(\Lambda,I)$ and $m\in \mathbb{N}$ where $M_k \to m M$ in the varifold sense. There exist at most $I$ points $\mathcal{Y} = \{y_i\}\subset M$ where the convergence to $M$ is smooth and graphical (with multiplicity $m$) away from $\mathcal{Y}$.

Moreover, associated with each $y \in \mathcal{Y}$ there exists a finite number $0<J_y\in \mathbb{N}$ of properly embedded minimal hypersurfaces $\{\Sigma_\ell^y\}_{\ell=1}^{J_y} \emb \R^{n+1}$ with finite total curvature for which 
\begin{equation*}
\lim_{k\to \infty} \A(M_k) = m\A(M) + \sum_{y\in \mathcal{Y}}\sum_{\ell=1}^{J_y} \A(\Sigma^y_\ell),
\end{equation*}
where $\sum_y J_y \leq I$. Furthermore, when $k$ is sufficiently large, the $M_k$'s are all diffeomorphic to one another. 
\end{theorem}

\begin{remark}
All conclusions of Theorem \ref{thm:main} remain valid if $\{M_k^n\}$ is a sequence of minimal hypersurfaces in $N$ with bounded area and bounded total curvature (as these assumptions imply bounded index by \cite{EM08, CT94}). In particular, the curvature quantisation result obtained here is also new in this setting. For the case $n=2$, one can consult papers of Brian White \cite{W15} and Antonio Ros \cite{R95} to see similar results for sequences of bounded total curvature and area. 

We also point out (we leave the details as an exercise), that it is possible to prove the following when $n=2$ using the Gauss equations and the Gauss-Bonnet formula: with the set-up as in the theorem above, for $k$ sufficiently large (regardless of whether $M_k$ or $M$ are orientable), we have
$$\chi(M_k) = m\chi(M) + \frac{1}{2\pi}\sum_y\sum_{\ell =1}^{J_y} \int_{\Sigma^y_\ell} K^y_{\ell} \,\,\ed V_{\Sigma^y_{\ell}}.$$
This tells us exactly how the genus is dropping depending on the multiplicity $m$, the topology of the limit, and the topology and number of ends of the bubbles. Thus we could even replace total curvature with Gauss curvature when $n=2$. The point being that (up to a factor of two) the Gauss curvature equals $|A|^2$ up to an ambient sectional curvature term, but since the $M_k$'s are converging as varifolds to $V= m M$ then these ambient terms cancel each-other out in the limit. 

Furthermore, when $n=2$ we have that $$\A(\Sigma^y_\ell) =-2\int_{\Sigma^y_\ell} K^y_\ell = 8\pi s$$ for some $s\in \mathbb{N}$. Therefore if there exists some $\delta >0$ such that $\A(M_k) \leq 8\pi - \delta $ then $\mathcal{Y} =\emptyset$ and the convergence to $M$ is smooth and graphical everywhere (possibly with multiplicity).  
\end{remark}

An easy corollary of Theorem \ref{thm:main} gives us a qualitative uniform upper bound on topology and the total curvature in terms of index and area: 

\begin{cor}\label{cor:Betti}
Let $2\leq n \leq 6$ and $N^{n+1}$ be a smooth closed Riemannian manifold. Then, given constants $\Lambda\in \R$ and $I\in \N$, there exists some $C=C(N, \Lambda, I)$ such that for all $M\in \mathcal{M}(\Lambda, I)$ and integers $1\leq i \leq n$ 
\begin{equation*}
\A(M) + b_i(M)\leq C,
\end{equation*}
where $b_i(M)$ is the $i^{th}$ Betti number of $M$. Moreover, $\mathcal{M}(\Lambda, I)$ contains only finitely many elements up to diffeomorphism. Finally, given a sequence $\{M_k\}\subset \mathfrak{M}(N)$, we conclude
\begin{enumerate}
\item if $\h^n(M_k)\leq \Lambda$ and either $\A(M_k)\to \infty$ or $b_i(M_k)\to \infty$ for some $i$, then we must have $index(M_k)\to \infty$,
\item if $index(M_k)\leq I$ and either $\A(M_k)\to \infty$ or $b_i(M_k)\to \infty$ for some $i$, then we must have $\h^n(M_k)\to \infty$. 
\end{enumerate}
\end{cor}

In particular, Corollary \ref{cor:Betti} and the results from \cite{EM08, CT94} show that in the case of bounded area, a bound on the Morse index is (qualitatively) \emph{equivalent} to a bound on the total curvature.

We expect that under a positive curvature assumption on $N$ one should also obtain \emph{a-priori} area estimates assuming an index bound. In particular, the index should control both the area and the topology in this case. When $n=2$, this has been proved  in \cite{CKM15} (using two-dimensional methods) under the very weak assumption of positive scalar curvature of $N$, $R_N >0$. Thus under the restriction that $Ric_N >0$ and $n=2$, bounds on the index and on the topology are (qualitatively) equivalent independent of an area bound by the compactness theorem of Choi-Schoen \cite{CSc85}. 

We will prove Theorem \ref{thm:main} using a bubbling argument using regions with positive index. However, these regions will also satisfy $\lambda_1 \to - \infty$ (see Corollary \ref{nontriv}). This observation allows us to conclude the following alternative to the main theorem. 

\begin{theorem}\label{cor:lambda}
Let $2\leq n \leq 6$ and $N^{n+1}$ be a smooth closed Riemannian manifold. If $\{M_k^n\}\subset \mathcal{M}_p(\Lambda,\mu)$ for some fixed constants $p\in\N$ and $\Lambda,\mu \in \R$, independent of $k$, then all the conclusions of Theorem \ref{thm:main} still hold true except this time we have $|\mathcal{Y}|, \sum_y J_y\leq p-1$.
\end{theorem}

Combining this with the index estimates of Cheng-Tysk \cite{CT94}, we see that a bound on the area and $\lambda_p$ implies a bound on the index. That is, we have the following corollary.
 
\begin{cor}\label{cor:equivalence}
Given $p\in \N$ and $\Lambda, \mu \in \R$, then there exists $I=I(p,\Lambda, \mu, N) \in \N$ such that
\begin{equation*}
\mathcal{M}_p(\Lambda,\mu) \subseteq \mathcal{M}(\Lambda,I).
\end{equation*}
Furthermore, given a sequence $\{M_k\}\subset \mathfrak{M}(N)$ with $\h^n(M_k)\leq \Lambda$ then if $index(M_k)\to \infty$ then $\lambda_p(M_k)\to -\infty$ for all $p$. In particular, the eigenvalues cannot concentrate in any compact region of the real line.
\end{cor}

\begin{remark}
The converse of Corollary \ref{cor:equivalence} is trivially true -- it should be clear that $\mathcal{M}(\Lambda, I)\subset \mathcal{M}_{I+1}(\Lambda, 0)$. Hence, under an area bound, the condition that $\lambda_p$ is bounded away from $-\infty$ for some $p\in\N$ is \emph{equivalent} to a bound on the index (and also equivalent to a bound on the total curvature).
\end{remark}

\textbf{Acknowledgements.} Both authors would like to thank Andr\'e Neves for many useful discussions. We would also like to thank Otis Chodosh for his interest in our work, and in pointing out a mistake in an earlier draft. The project started whilst BS was supported by Professor Neves' European Research Council StG agreement number P34897, The Leverhulme Trust and the EPSRC Programme Grant entitled ``Singularities of Geometric Partial Differential Equations''™ reference number EP/K00865X/1. Since 1st October 2015, BS has held a junior visiting position at the Centro di Ricerca Matematica Ennio De Giorgi and would like to thank the centre for its support and hospitality. RB would like to thank the EPSRC for partially funding his research under grant number EP/M011224/1.

%------------------------------------------------------------------
\section{Localised compactness theory}\label{section:localcomp}

The goal of this section is to localise the compactness theory developed by the second author in \cite{Sha15}, in fact more precisely we write down a local version of the theory developed in \cite{ACS15}. But before stating our local compactness theorem, let us recall two important results from Schoen-Simon \cite{ScS81} which we will often use throughout this article. The first is a regularity theorem for \emph{stable} minimal hypersurfaces.

\begin{prop}[{\cite[Corollary 1]{ScS81}}]\label{ssreg}
Suppose $M\subset N$ is minimal and embedded. Let $p\in M$ and $\varrho > 0$ with $\h^n(M\cap B^N_{\varrho}(p))<\infty$, and suppose that $M$ is stable in $B^N_{\varrho}(p)$ with respect to the area functional. If $n\leq 6$ and $\varrho^{-n}\h^n(M\cap B^N_{\varrho}(p)) \leq \mu < \infty$ then there exists $C=C(N,\mu)<\infty$, such that
\begin{equation*}
\sup_{B^N_{\varrho/2}(p)} |A|\leq \frac{C}{\varrho}.
\end{equation*}
\end{prop}

The second result is a consequence of the monotonicity formula for minimal hypersurfaces.
\begin{prop}[{\cite[pp 778--780]{ScS81}}]\label{ssvol}
Suppose $M\subset N$ is minimal and embedded with $\h^n(M)\leq \Lambda$. There exists some $0<\delta_i=\delta_i(N) < \frac{inj_N}{2}$ and $C=C(N, \Lambda)$ such that for $p\in M$ and any $q\in B^N_{\delta_i}(p)\cap M$, $0< r<\delta_i$, we have 
\begin{equation*}
\frac{1}{C} \leq \frac{\h^n(M\cap B^N_r(q))}{r^n} \leq C.
\end{equation*}
\end{prop}
\begin{remark}\label{ssregr}
Notice that if $M\subset N$ is minimal with $\h^n(M) \leq \Lambda$ then Propositions \ref{ssreg} and \ref{ssvol} imply that if $B_{\varrho}(p)\cap M$ is stable for $p\in M$ and $0<\varrho< \delta_i$, then there exists some $C=C(N,\Lambda)$ such that
\begin{equation*}
\sup_{B^N_{\varrho/2}(p)} |A|\leq \frac{C}{\varrho}.
\end{equation*}  
\end{remark}

Now, to state the localised compactness theorem, we first collect some notation. Let $p\in M$. Then by choosing normal coordinates, we can identify the ball $B^N_{\varrho}(p)$ with $B^{\R^{n+1}}_{\varrho}(0)$ and assume that $T_pM = \R^n = \{x^{n+1}=0\} \subset \R^{n+1}$. By abuse of notation, we will often simply write $B_\varrho$ to denote both $B^N_{\varrho}(p)$ and $B^{\R^{n+1}}_{\varrho}(0)$. We also define the cylinder $C_{\varrho} = B^n_{\varrho}(0) \times \R$, where $B^n_{\varrho}(0)$ denotes a ball in $\R^n = T_pM$. 

Using the above notation, we say that $M_k\to M$ smoothly and graphically at $p\in M$ if for all sufficiently large $k$, there exists smooth functions $u_k^1,\ldots,u_k^m:B^n_{\varrho}(0) \to \R$ such that $M_k\cap C_{\varrho}$ is the collection of graphs of the $u_k^i$ and $u_k^i\to u_M$ in $C^k$ for all $k\geq 2$, where $u_M$ is a graph describing $M$. We call $m$ the \emph{number of leaves} or the \emph{multiplicity} of the smooth and graphical convergence.

Note that if $M_k \to M$ smoothly and graphically, we can also find $\varrho>0$ such that we can consider the $u_k^i$ to be defined on $M\cap B^N_{\varrho}(p)$ with $u_k^i\to 0$ in $C^k$ for all $k\geq 2$. Furthermore, also note that if the convergence is smooth and graphical away from a finite set $\mathcal{Y}$ and $M$ is connected and embedded (so that $M\setminus \mathcal{Y}$ is also connected) then the number of leaves in the convergence is a constant over $M\setminus \mathcal{Y}$. 

Finally, similarly to the class $\mathcal{M}(\Lambda,I) \subset \mathfrak{M}(N)$, we define the class $\mathfrak{M}(U)$ of connected, smooth and embedded minimal hypersurfaces $M\subset U$ where $U$ is an open contractible set of $\R^{n+1}$. We also define the subset $\mathcal{M}(\Lambda,I,U)$, defined by
\begin{equation*}
\Big\{M\in \mathfrak{M}(U) \mid \tfrac{\h^n(M\cap B_R(x))}{R^n}\leq \Lambda \text{ for all $x\in M$, $B_R(x)\subset U$, and $index(M)\leq I$}\Big\}.
\end{equation*}
For simplicity we will actually deal with the class $\mathcal{M}_p(\Lambda,\mu, U)$ defined by
\begin{equation}
\Big\{ M \in \mathfrak{M}(U) \mid \tfrac{\h^n(M\cap B_R(x))}{R^n}\leq \Lambda \text{ for all $x\in M$, $B_R(x)\subset U$, and $\lambda_p(M) \geq -\mu$} \Big\}.
\end{equation}
In the above, both $index(M)$ and $\lambda_p(M)$ are computed with respect to Dirichlet boundary conditions on $\partial(M\cap U)$. Below, we state the local compactness theory only considering $\mathcal{M}_p (\Lambda, \mu)$, but equivalent statements concerning $\mathcal{M}(\Lambda, I)$ can then be obtained directly, as $\mathcal{M}(\Lambda,I,U)\subset \mathcal{M}_{I+1}(\Lambda, 0,U)$ by definition.

The localised compactness theorem can be stated as follows.

\begin{theorem}[{Local version of \cite[Theorem 1.3]{ACS15}}]\label{thm:oldmain}
Let $2\leq n \leq 6$ and $\{U_k^{n+1}\} \subset \R^{n+1}$ be smooth, simply connected, open and bounded subsets of $\R^{n+1}$ equipped with metrics $g_k$. We will assume that $0\in U_{k}\subset U_{k+1}$ and for any compact $V\cemb U_k$ (for $k$ sufficiently large), we have that $g_k\to g$ smoothly for some limit metric $g$. If $\{0\in M_k^n\}\subset \mathcal{M}_p(\Lambda,\mu, U_k)$ for some fixed constants $\Lambda, \mu\in \R$, $p\in \N$ independent of $k$, then for any open $V$ such that $V\cemb U_k$ for $k$ sufficiently large, there exists a closed connected and embedded minimal hypersurface $M\in \mathcal{M}_p(\Lambda,\mu,V)$ where $M_k \to M$ in the varifold sense on $V$.
 
Now, assuming that $M_k\neq M$ eventually, we have that the convergence is locally  smooth and graphical for all $x\in M\setminus  \mathcal{Y}$ where $\mathcal{Y}=\{y_i\}_{i=1}^J\subset M$ is a finite set with $J\leq p-1$. Moreover, if the number of leaves in the convergence is one then $\mathcal{Y}=\emptyset$. 

Finally, $\mathcal{Y}\neq \emptyset$ implies that $M$ must be stable in any proper subset of $V$. 
\end{theorem}

\begin{proof}
The reader can follow the steps of the (global) compactness theorem derived in \cite{ACS15} in order to prove all but the final statement of the theorem. Note that $M$ must be two-sided since we could always choose $V$ to be simply connected, and $M$ is embedded. 

If the convergence is not smooth then it must be of multiplicity, and hence as in \cite{Sha15} we can construct a nowhere vanishing Jacobi field $f>0$ with respect to some choice of unit normal $\nu$ defined on $W\cap M$ for any $W\cemb V$. Thus, if $M$ is unstable on some proper subset $S$ there exists a positive function $\phi$ which vanishes on $\partial (S\cap M)$ (and is strictly positive on the interior)  associated with the smallest eigenvalue $\lambda_1 (S\cap M) < 0$. Now, let $f$ be the Jacobi function associated with $S\cap M$ -- notice that $f>0$ even at the boundary. By multiplying $f$ by some constant, we can ensure that $f\geq \phi$, and moreover $f(p)=\phi(p)$ for some $p\in S\cap M$. Thus $f-\phi \geq 0$, it attains its minimum at $p$ (on the interior of $S\cap M$) and 
\begin{align*}
\Delta_M (f - \phi) (p) &= - \big(|A|^2 + Ric_N(\nu,\nu)\big) (f - \phi) (p)+ \lambda_1(S\cap M) \phi( p)\\
&= \lambda_1(S\cap M) \phi( p) <0,
\end{align*}
a contradiction.
\end{proof}
 
This local version of the compactness theorem has the following useful consequence for us.
\begin{cor}\label{cor:Rn}
Let $2\leq n \leq 6$ and $\{U_k^{n+1}\}\subset \R^{n+1}$ be smooth, simply connected, open and bounded subsets of $\R^{n+1}$ with $0\in U_{k}\subset U_{k+1}$ and $\R^{n+1} = \bigcup_k U_k$. Again, suppose that each $U_k$ is equipped with a metric $g_k$ which converges smoothly and locally to some metric $g$ defined on $\R^{n+1}$.  
  
If $\{0\in M_k^n\}\subset \mathcal{M}_p(\Lambda,\mu, U_k)$ for some fixed constants $\Lambda, \mu \in \R$, $p\in \N$ independent of $k$ and $\liminf\lambda_p(M_k)\geq 0$. Then there exists some $M\in \mathcal{M}(\Lambda, p-1,  \R^{n+1})$ with $M_k \to M$ smoothly and graphically away from at most $p-1$ points $\mathcal{Y}=\{y_i\}_{i=1}^J$ on any open $V\cemb \R^{n+1}$. 

If $g = g_0$ is the Euclidean metric, then $M$ has finite total curvature $\mathcal{A}(M)<\infty$ and if $\mathcal{Y}\neq \emptyset$ then $M$ must be a plane. 
\end{cor}

\begin{proof}
Again, most of the statements are immediate. By Theorem \ref{thm:oldmain}, we know that for any relatively compact $V$ we have the existence of some $M\in \mathcal{M}_p(\Lambda,\mu,V)$ as above. Since this is true for all $V$, we end up with $M\in \mathcal{M}_p(\Lambda, \mu ,\R^{n+1})$ and now we check that in fact $index(M)\leq p-1$. For a contradiction we suppose the contrary; this guarantees that on some such set $V$ we have $\lambda_p(M\cap V)=\alpha<0$ and following the proof in \cite[p.~8]{ACS15} we find eventually that $\lambda_p(M_k\cap V)\leq \frac{\alpha}{2}$ contradicting the assumption that $\liminf\lambda_p(M_k)\geq 0$.

If $g=g_0$ then $M$ has finite total curvature by the main result in Tysk \cite{T89}. If also $\mathcal{Y}\neq \emptyset$ then it is stable on every compact set (and therefore stable), it has finite Euclidean volume growth and therefore by Proposition \ref{ssreg} it is a plane.
\end{proof}

A further important consequence of the above is the following corollary.

\begin{cor}\label{nontriv}
Let $\{M_k\}\subset \mathcal{M}_p(\Lambda, \mu) \subset \mathfrak{M}(N)$ so that (up to subsequence) $M_k\to M$ for $M$ as in \cite{ACS15}. Suppose that there are sequences $\{r_k\}\subset \R_{>0}$ and $\{p_k\}\subset N$ satisfying
\begin{itemize}
\item $M_k\cap B_{r_k + (r_k)^2}(p_k)$ is unstable for all $k$,
\item $p_k \in M_k$ with $p_k \to y\in M$ and $r_k \to 0$,
\item $M_k\cap B_{r_k/2}(q_k)$ is stable for all $q_k\in M_k\cap B_{\varrho_k}(p_k)$ with $\varrho_k/r_k\to \infty$. 
\end{itemize}
Then some connected component of the re-scaled hypersurfaces 
$$\tilde{M}_k \cap B_{\varrho_k/r_k}(0) : = \frac{1}{r_k} \big(M_k\cap B_{\varrho_k}(p_k)- p_k\big)$$ 
must converge smoothly (on any compact set) with multiplicity one to a non-planar minimal hypersurface $\Sigma_1\subset \R^{n+1}$ with $\Sigma_1\cap B_2(0)$ unstable. As a consequence, we must have that $\lambda_1(M_k \cap (B_{2r_k}(p_k)))\to -\infty$.

If we furthermore suppose that there is another point-scale sequence $(\hat{r}_k, \hat{p}_k)$ with $\hat{r}_k \geq r_k$ for all $k$ and satisfying
\begin{itemize}
\item $M_k\cap (B_{\hat{r}_k+ (\hat{r}_k)^2}(\hat{p}_k)\setminus B_{2r_k}(p_k))$ is unstable for all $k$,
\item there exists some $C<\infty$ with $B_{\hat{r}_k}(\hat{p}_k)\subset B_{Cr_k}(p_k)$ for all $k$,
\end{itemize}
then $\lambda_1(M_k\cap (B_{2\hat{r}_k}(\hat{p}_k)\setminus B_{2r_k}(p_k))) \to -\infty.$
\end{cor}

\begin{proof}
We use Corollary \ref{cor:Rn} in order to conclude that (since $\lambda_p(\tilde{M}_k\cap B_{\varrho_k/r_k}(0))\to 0$) for every $R$ each component of $\tilde{M}_k\cap B_R(0)$ converges smoothly and graphically to some $\Sigma_{\ell}\in \mathcal{M}(\Lambda, p-1, \R^{n+1})$ away from at most finitely many points $\mathcal{Y}$. However, we get that $\mathcal{Y}=\emptyset$ since $\tilde{M}_k$ is stable on all balls of radius $\frac12$. Moreover each $\Sigma_{\ell}$ is minimal with respect to the Euclidean metric (since we are blowing up) and has finite volume growth by the monotonicity formula, Proposition \ref{ssvol}. 

Thus, each component converges smoothly on every compact subset of $\R^{n+1}$ with multiplicity one to some $\Sigma_{\ell}$. If $\Sigma_{\ell}$ is always planar for each connected component, then in particular every connected component of $\tilde{M}_k\cap B_2$ converges smoothly to a flat disk. At which point we may conclude that the first eigenvalue of the Laplacian $\mu_k$ on $\tilde{M}_k \cap B_2$ is converging to the first eigenvalue of the Laplacian on a flat $B_2\subset \R^n$ (eigenvalues with Dirichlet boundary conditions). In particular $\mu_k >0$ eventually. Thus we must have, for any non-zero $\phi\in C_c^{\infty}(B_2)$
\begin{multline*}
\int_{\tilde{M}_k\cap B_2} \Big(|\D \phi|^2 - (|\tilde{A}_k|^2 + Ric_k(\nu_k,\nu_k))\phi^2 \Big)\ed V_{\tilde{M}_k}\\
\geq \big(\mu_k - \sup_{\tilde{M}_k} (|\tilde{A}_k|^2 + Ric_k(\nu_k,\nu_k))\big) \int_{\tilde{M}_k\cap B_2} \phi^2 \id V_{\tilde{M}_k} >0
\end{multline*}
when $k$ is sufficiently large. In particular, $\tilde{M}_k \cap B_2$ is stable, a contradiction. 

Now suppose that $\Sigma_{\ell}\cap B_2$ is stable for all $\ell$. Thus $\lambda_1(\Sigma_\ell\cap B_{3/2}) > 0$. Since $\tilde{M}_k\cap B_{3/2}$ converges smoothly with multiplicity one to $\Sigma_{\ell} \cap B_{3/2}$ we can smoothly compare their second variations; i.e. for any $\phi\in C_c^{\infty}(\tilde{M}_k\cap B_{3/2})$ we can push it forward to give $\phi \in C_c^{\infty}(\Sigma_\ell\cap B_{3/2})$ (without re-labelling) and 
$$\int_{\Sigma_{\ell}\cap B_{3/2}} \! \Big(|\D \phi|^2 - |A_{\ell}|^2\phi^2\Big) \ed V_{\Sigma_\ell} - \int_{\tilde{M}_k\cap B_{3/2}} \! \Big( |\D \phi|^2 - (|\tilde{A}_k|^2 + Ric_k(\nu_k,\nu_k)) \phi^2 \Big) \ed V_{\tilde{M}_k} \to 0.$$ 
In particular $\tilde{M}_k \cap B_{3/2}$ is strictly stable eventually, a contradiction. Thus we must have that $\lambda_1(\tilde{M}_k\cap B_2)$ is strictly bounded away from zero (otherwise $\Sigma_1\cap B_2$ would indeed be stable by the arguments in \cite[p.~8]{ACS15}) and a re-scaling argument finishes the first part of the proof. 

The argument for the second part is similar; at the $r_k$ scale we have that $B_{\hat{r}_k}(\hat{p}_k)$ corresponds to a ball $B_{\tilde{r}_k}(\tilde{p}_k)$ with $1\leq \tilde{r}_k \leq C$ and $\tilde{p}_k\in B_C(0)$. At which point we can conclude that $\tilde{r}_k \to \tilde{r}\in [1,C]$ and $\tilde{p}_k \to \tilde{p}\in B_C(0)$. Assuming that $(B_{2\tilde{r}}(\tilde{p})\setminus B_2(0))\cap \Sigma_{\ell}$ is stable, the reader can follow the argument above to find a contradiction and finish the proof.  
\end{proof}

%------------------------------------------------------------------
\section{The bubbling argument}\label{sect.bubbling}

In this section, we start our proof of the main theorem. Take the set-up as described in Theorem \ref{thm:main} and assume that $k$ is sufficiently large and $\delta$ sufficiently small so that 
\begin{equation*}
2\delta<\min\left\{\inf_{\mathcal{Y}\ni y_i\neq y_j\in \mathcal{Y}} d_g(y_i,y_j),\,\,\, \frac{inj_N}{2}\right\}.
\end{equation*}

We know that the first part of Theorem \ref{thm:main} holds true by \cite{Sha15} and if the multiplicity of the convergence is equal to $m$ we must have
\begin{equation*}
\int_{M_k\setminus \left(\bigcup_{y_i \in \mathcal{Y}} B_{\delta}(y_i)\right)} |A_k|^n \id V_{M_k} \to m \int_{M\setminus \left(\bigcup_{y_i \in \mathcal{Y}} B_{\delta}(y_i)\right)} |A|^n \id V_M.
\end{equation*}
This also holds true for all $\delta >0$ so that in fact 
\begin{equation}\label{eq.mainlimit}
\lim_{\delta\to 0}\lim_{k\to\infty}\int_{M_k\setminus \left(\bigcup_{y_i \in \mathcal{Y}} B_{\delta}(y_i)\right)} |A_k|^n \id V_{M_k} \to m \mathcal{A}(M).
\end{equation}
Thus it remains to understand what is happening on the small balls $B_\delta (y_i)$ as $\delta \to 0$. In order to study the behaviour here, we need to consider bubbles by extracting various point-scale sequences and look at their blow-up limits. During this process, we often apply the local compactness theorem from the previous section which requires us to (iteratively) pass to subsequences, but for simplicity, we won't always state this explicitly. The important point here is that there will only be finitely many steps where we pass to a subsequence of a previous subsequence, so no diagonal argument is needed.

For fixed $y\in \mathcal{Y}$ consider $M_k\cap B_\delta(y)$, it is possible that there is more than one component but for $k$ sufficiently large there are at most $m$. Note that by the choice of $y$ we must have $M_k \cap B_{r}(y)$ is unstable for all fixed $r>0$ and $k$ large enough. The rough plan from here is to extract point-scale subsequences $(p_k, r_k)$, $p_k\to y$ and $r_k\to 0$ so that we capture some index in each $B_{r_k}(p_k)\cap M_k$. Moreover a standard bubbling argument will tell us that we can capture all the coalescing index in such a way. We will thus end up considering at most $I$ point-scale sequences, each capturing some index.

To be more precise, we will prove the following bubbling theorem. 

\begin{theorem}\label{thm:bubbles}
Let $2\leq n \leq 6$ and $N^{n+1}$ be a smooth closed Riemannian manifold. If $\{M_k^n\}\subset \mathcal{M}(\Lambda,I)$ for some fixed constants $\Lambda\in \R$, $I\in \N$ independent of $k$, then up to subsequence, there exists $M\in \mathcal{M}(\Lambda,I)$ where $M_k \to M$ in the varifold sense. There exist at most $I$ points $\mathcal{Y} = \{y_i\}\subset M$ where the convergence to $M$ is smooth and graphical (with multiplicity $m$) away from $\mathcal{Y}$. Furthermore, for each $y\in \mathcal{Y}$ there exist a finite number of point-scale sequences $\{(p^i_k, r^i_k)\}_{i=1}^{L_y}$ where $\sum_{y\in \mathcal{Y}} L_y \leq I$ with $M_k \ni p^i_k \to y$, $r^i_k\to 0$, and a finite number $\{\Sigma_i\}_{i=1}^{L_y}$ of embedded minimal hypersurfaces in $\R^n$ with finite total curvature where the following is true.
\begin{enumerate}

\item 
 
\begin{itemize}
\item Each $\Sigma_i$ is a disjoint union of planes and at least one non-trivial minimal hypersurface.
\item For all $i\neq j$ 
$$\frac{r^i_k}{r^j_k} + \frac{r^j_k}{r^i_k} + \frac{dist_g(p^i_k, p^j_k)}{r^i_k + r^j_k} \to \infty.$$
\item Taking normal coordinates centred at $p^i_k$ and letting 
$$\tilde{M}^i_k:=\frac{M_k}{r^i_k} \subset \R^{n+1}$$ then $\tilde{M}^i_k$ converges locally smoothly and graphically to $\Sigma_i$ away from the origin, and whenever a component converges to a non-trivial component of $\Sigma_i$ then it does so with multiplicity one. 
\item Given any other sequence $M_k \ni q_k$ and $\varrho_k \to 0$ with $q_k \to y$ and 
$$\min_{i=1,\dots L_y} \Big(\frac{\varrho_k}{r^i_k} + \frac{r^i_k}{\varrho_k} + \frac{dist_g(q_k,p^i_k)}{\varrho_k + r^i_k}\Big)\to \infty$$
then taking normal coordinates at $q_k$ and setting 
$$\hat{M}_k:= \frac{M_k}{\varrho_k}\subset \R^{n+1}$$
then $\hat{M}_k$ converges to a collection of parallel planes. 
\item Furthermore, when $n=2$, $\Sigma_i$ is always connected. The convergence is smooth, and of multiplicity one locally. Moreover we always have 
$$\frac{dist_g(p^i_k,p^j_k)}{r^i_k + r^j_k}\to \infty.$$
\end{itemize}
\item 
Finally there is no lost total curvature in the limit; if we collect all of the $\Sigma_{\ell}$ together for each $y\in \mathcal{Y}$ (without re-labelling) then we have 
\begin{equation}\label{eq.curvid}
\lim_{k\to \infty} \A(M_k) = m\A(M) + \sum_{\ell=1}^J \A(\Sigma_\ell),
\end{equation}
where $J \leq I$. When $k$ is sufficiently large, the $M_k$'s are all diffeomorphic to one another.
\end{enumerate}
\end{theorem}

Obviously, Theorem \ref{thm:bubbles} yield the result claimed in Theorem \ref{thm:main}. In this section, we will prove the first part of Theorem \ref{thm:bubbles} which -- together with \eqref{eq.mainlimit} -- also yields the quantisation result in \eqref{eq.curvid} modulo a set of ``neck regions''. In fact, in the argument presented here, these are precisely the regions where we have no a-priori control on the total curvature. However, by construction, our minimal hypersurfaces will be uniformly graphical with uniformly small slope and small dyadic total curvature. Using these facts we will control these intermediate regions in the next section. For convenience, we give a precise definition of how these neck regions look like first.

\begin{defn}\label{neckdef}
Let $M\subset \mathcal{M}(\Lambda, I)$. For $\delta<\delta_i$, $p\in M$ we say that $M\cap B^{\R^{n+1}}_\delta(p)\setminus B^{\R^{n+1}}_{\eps}(p)$ is a \emph{neck of order} ($\eta, L$) if we have $\eps < \delta/4$ and $M\cap (B_\delta\setminus B_{\eps})$ is uniformly graphical over some plane which we may assume (after a rotation) to be defined by $\{x^{n+1}=0\}$ in normal coordinates about $p$. More precisely there exist functions $u^1,\dots u^L$ such that 
\begin{equation*}
M\cap (B_\delta \setminus B_{\eps}) = \bigcup_{i=1}^L \; \{(x^1,\dots , x^n, u^i(x^1,\dots, x^n))\}
\end{equation*}
and also 
\begin{equation}
\sup_{i=1,\ldots,L} \|\D u^i\|_{L^{\infty}} + \sup_{\eps< \varrho < \frac{\delta}{2}} \int_{M \cap (B_{2\varrho}\setminus B_\varrho)} |A|^n \id V_M = :\eta < \infty. 
\end{equation}
\end{defn}

\begin{proof}[Proof of Theorem \ref{thm:bubbles} Part 1.]
Let us start the proof of the bubbling theorem by picking the first point-scale sequence as follows: Let
$$r^1_k = \inf \Big\{r>0 \mid M_k\cap B_r(p) \text{ is unstable for some } p\in B_\delta(y)\cap M_k \Big\}.$$
Note that with $r^1_k$ defined above, we can pick $p^1_k\in B_\delta(y)\cap M_k$ and $\delta >r^1_k>0$ such that $M_k\cap B_{r^1_k+(r^1_k)^2}(p^1_k)$ is unstable. 

Clearly we must have $r^1_k\to 0$ and $p^1_k\to y$ as otherwise the regularity theory of Schoen-Simon (Proposition \ref{ssreg}) would give a uniform $L^{\infty}$ estimate on the second fundamental form in $B_\delta(y)$ and we reach a contradiction to the fact that $y$ is a point of bad convergence. 

Now we extract a second point-scale sequence, but in order to do so we define an admissible class of balls on which we look. Define, for $r^1_k \leq r<\delta$
\begin{align*}
\mathfrak{C}^2_k:=\Big\{&B_r(p) \mid M_k\cap \Big(B_r(p)\setminus B_{2r^1_k}(p^1_k) \Big) \,\,\text{is unstable and $p\in B_\delta (y)\cap M_k$. Furthermore } \\
&\text{ if $B_{2r}(p)\cap B_{2r^1_k}(p^1_k)\neq\emptyset$ then $B_{6r}(p)\cap M_k$ consists of at least two components}\\
&\text{$C^k_1$, $C^k_2$ with $p\in C^k_2$, $B_r(p)\cap C^k_2$ is unstable and $C^k_2\cap B_{2r^1_k}(p^1_k) = \emptyset$.}\Big\}
\end{align*}
If $\mathfrak{C}^2_k\neq \emptyset$ we pick the second point-scale sequence $(p^2_k,r^2_k)$ so that $B_{r^2_k + (r^2_k)^2 }(p^2_k)\in \mathfrak{C}^2_k$, so that in particular  $M_k\cap \big(B_{r^2_k + (r^2_k)^2 }(p^2_k)\setminus B_{2r^1_k}(p^1_k)\big)$ is unstable, and $p^2_k\in M_k$ and
\begin{eqnarray*}
r^2_k = \inf \Big\{r>0 \mid  B_r(p)\in \mathfrak{C}^2_k\Big\}. \end{eqnarray*}
 
Note that by construction we must have $r^2_k\geq r^1_k$. We ask first whether $r^2_k\to 0$. If not, we discard $(p_k^2,r_k^2)$ and we can find no more point-scale sequences; the process stops. Similarly if $\mathfrak{C}^2_k = \emptyset$ we also stop.  

If we do have $r^2_k\to 0$, then we ask whether or not 
\begin{equation}\label{eq.psinf}
\limsup_{k\to \infty}  \bigg(\frac{r^2_k}{r^1_k} + \frac{dist_g(p_k^1,p_k^2)}{r^2_k} \bigg)= \infty.
\end{equation}
Regardless of the outcome of the above, we still continue to try to find the next point-scale sequence. However, if \eqref{eq.psinf} is not true (that is, if $(r^2_k/r^1_k + dist_g(p_k^1,p_k^2)/r^2_k)$ stays bounded) then we mark this sequence (to be ignored later) since it must be part of the previous point-scale sequence (see the paragraph below for more details). In either case, we still keep track of the region $B_{2r^2_k}(p^2_k)$ in order to find the next point-scale sequence, regardless of whether or not the sequence was marked. 

Suppose we have extracted $\ell-1$ point-scale sequences (including the ones that we will ignore) for $y$. Then we  continue (or not) under the following rules: let $U^{\ell-1}_k = \bigcup_{s=1}^{\ell-1} B_{2r^s_k }(p_k^s)$ and define an admissible class of balls with $r^{\ell -1}_k\leq r<\delta$ and
\begin{align*}
\mathfrak{C}^\ell_k:=\Big\{&B_r(p) \mid M_k\cap \Big(B_r(p)\setminus U^{\ell -1}_k \Big) \,\,\text{is unstable and $p\in B_\delta (y)\cap M_k$. Furthermore } \\
&\text{whenever $B_{2r}(p)\cap U^{\ell-1}_k\neq\emptyset$ then $B_{6r}(p)\cap M_k$ consists of at least two components}\\
&\text{$C^k_1$, $C^k_2$ with $p\in C^k_2$, $B_r(p)\cap C^k_2$ is unstable and $C^k_2\cap U^{\ell -1}_k =\emptyset$}\Big\}
\end{align*}
Now if $\mathfrak{C}^\ell_k =\emptyset$ we stop. If not, pick the next point-scale sequence $(p^\ell_k, r^\ell_k)$ in such a way that $B_{r^\ell_k +(r^{\ell}_k)^2}(p^\ell_k)\in \mathfrak{C}^\ell_k$ (in particular $M_k \cap\big( B_{r^\ell_k +(r^{\ell}_k)^2}(p^\ell_k)\setminus U^\ell_k\big)$ is unstable) and
\begin{equation*}
r^\ell_k = \inf \Big\{r>0 \mid B_r(p)\in \mathfrak{C}^{\ell}_k.\Big\}.
\end{equation*}
Again, we must have $r^\ell_k\geq r^{\ell-1}_k$. As before, we ask ourselves at most two questions: First, does $r^\ell_k \to 0$? If not, then we discard $(r_k^\ell,p_k^\ell)$ and the process stops. If $r_k^\ell \to 0$, we ask whether or not it is true that
\begin{equation}\label{eq.psinf2}
\min_{i=1,\dots, \ell-1}\limsup_{k\to \infty}  \bigg(\frac{r^\ell_k}{r^{i}_k} + \frac{dist_g(p_k^\ell,p_k^i)}{r^\ell_k} \bigg) = \infty.
\end{equation}
Again, regardless of whether this is true or not, we continue -- but if it is not true, then we mark the point-scale sequence $(p_k^\ell,r_k^\ell)$ to be ignored later. 

We continue to do this until we exhaust all such sequences. Note that each time we pick a sequence (regardless of whether we intend to ignore it or not) we are accounting for some region where we have index -- and all of these regions are disjoint by construction. Thus, the process must stop after at most $\ell \leq I$ iterations (see \cite[Lemma 3.1]{Sha15}).  

First we remark on the `ignored' sequences $(r^s_k, p^s_k)$ -- these are the ones for which $r^s_k \to 0$ but there exists some $C<\infty$ and $t \in 1,\dots, s-1$ such that
$$ \frac{r^s_k}{r^{t}_k} + \frac{dist_g(p_k^s,p_k^t)}{r^s_k}\leq C.$$
Thus, there exists $C<\infty$ such that $B_{r^s_k}(p^s_k)\subset B_{Cr^t_k}(p^t_k)$. In this sense, we will see these sequences as part of the same point-scale sequence -- if we blow up at either scale $s$ or $t$ (as described below) both re-scaled regions remain a bounded distance from each other. Notice that $t$ itself might be ignored, however in this case we just go one step further and eventually we find some $h\in 1,\ldots, t-1$ such that this scale is not ignored and we have $B_{r^s_k}(p^s_k)\subset B_{Cr^h_k}(p^h_k)$ for some $C<\infty$.

\textbf{Blowing up: The first bubble.} Now consider the first scale $(r_k^1, p_k^1)$; notice that $p_k^1\to y$, so by our choice of $\delta$ we can pick normal coordinates for $N$ centred at $p_k^1$ on a geodesic ball of radius $\delta$. We now consider $M_k\cap B_{\delta}(p_k)$ -- or equivalently $M_k \cap B^{\R^{n+1}}_{\delta}(0)$, where the metric on $N$ in these coordinates is given as $g_k = g_0 + O_k(|x|^2)$. Here, $g_0$ denotes the Euclidean metric. Notice that we actually have $O_k(|x|^2) = O(|x|^2)$. Now consider the sequence 
$$\tilde{M}^1_k : = \frac{1}{r^1_k} (M_k\cap B_\delta) \subset \Big(B^{\R^{n+1}}_{\delta/r^1_k}(0), \tilde{g}_k\Big)$$ 
where $\tilde{g}_k = g_0 + (r_k^1)^2 O(|x|^2)$. By the choice of $r^1_k$ we have that $\tilde{M}^1_k$ is stable inside every ball of radius $\frac12$ in $B^{\R^{n+1}}_{\delta/r^1_k}$, and moreover 
\begin{equation}\label{mon}
\frac{\h^n(\tilde{M}^1_k \cap B_R)}{R^n} =  \frac{\h^n(M_k\cap B_{Rr^1_k}(p^1_k))}{(Rr^1_k)^n} \leq C
\end{equation}
 for some uniform constant $C$. This follows from the monotonicity formula for $N$ which holds at small scales (see Section \ref{section:localcomp} and e.g. \cite[(1.18)]{ScS81}).
   
Thus, by setting 
\begin{equation*}
U_k=B^{\R^{n+1}}_{\delta/r^1_k}(0), 
\end{equation*}
we have that $\{U_k\}$ is a nested exhaustion of $\R^{n+1}$ and that $0\in \tilde{M}^1_k \in \mathcal{M}(\Lambda,I, U_k)$. So by the local compactness theory, or more precisely Corollary \ref{cor:Rn}, there exists $\Sigma_1 \in \mathcal{M}(\Lambda, I, \R^{n+1})$ such that $\tilde{M}^1_k$ converges (up to a subsequence) smoothly and graphically to $\Sigma_1$ on any compact region $V\subset \R^{n+1}$ away from possibly a finite set $\mathcal{Y}$. But since $\tilde{M}^1_k$ is stable on every ball of radius less than $\frac12$, we must have that $\mathcal{Y} =\emptyset$ and therefore the convergence is locally smooth and of multiplicity one on every such $V$. Moreover $\tilde{g}_k \to g_0$ smoothly and uniformly on every $V$, so that $\mathcal{A}(\Sigma_1)<\infty$ by a result of Tysk \cite{T89}.

In particular we have (by the scale-invariance of the total curvature)
\begin{align*}
\lim_{R\to \infty} \lim_{k\to \infty} \int_{M_k\cap B_{Rr^1_k}(p^1_k)} |A_k|^n \id V_{M_k} &= \lim_{R\to \infty} \lim_{k\to \infty} \int_{\tilde{M}^1_k\cap B_{R}(0)} |\tilde{A}^1_k|^n \id V_{\tilde{M}^1_k}\\
&= \lim_{R\to \infty} \int_{\Sigma_1 \cap B_R(0)} |A|^n \id V_{\Sigma_1}\\
&= \mathcal{A}(\Sigma_1).
\end{align*}

\begin{remark}\label{remark:Sigma1}
By Corollary \ref{nontriv}, we know that $\Sigma_1$ cannot be trivial. We also point out that when $n=2$, $\Sigma_1$ must be connected. This follows since if it is disconnected, then we get convergence to two complete minimal hypersurfaces with finite total curvature which do not intersect -- but this contradicts the half-space theorem \cite{HM90}. If $n\geq 3$, it is conceivable that $\Sigma_1$ is disconnected (since even non-trivial minimal hypersurfaces lie in slabs) -- however the number of components is clearly bounded by the volume growth and the end result is that we still have finite total curvature. 
\end{remark}

\textbf{The second bubble.} Now, we need to worry about all the other sequences. The following ideas are standard in bubbling theory, but for convenience of the reader we explain the main steps. This is the point where we avoid all of the marked point-scale sequences and re-label the remaining ones to be  consecutive $1,\dots , \ell$. Moreover, by passing to a further subsequence (of the indices $k$) if necessary, we may assume that the limes superior in \eqref{eq.psinf} and \eqref{eq.psinf2} is actually a limit. In particular, for $(r_k^2, p_k^2)$ we must have 
\begin{equation}\label{step1}
\lim_{k\to \infty} \bigg( \frac{r^2_k}{r^1_k} + \frac{dist_g(p_k^1,p_k^2)}{r^2_k} \bigg)= \infty.
\end{equation}

We have two cases. 

\textbf{Case 1: The second fraction in \eqref{step1} tends to infinity -- the bubbles are forming separately.} In this case, the reader can check that if we blow up for $(r_k^2, p_k^2)$ in a similar way as for $(r_k^1, p_k^1)$ above we get the same conclusion (the first bubble will disappear off at infinity). More precisely, let $\varrho_k:=  dist_g(p^1_k,p^2_k)$ and consider the geodesics ball $B_{\varrho_k}(p^2_k)$ with 
$$0\in\tilde{M}_k^2 = \frac{M_k}{r^2_k}\subset B^{\mathbb{R}^{n+1}}_{\varrho_k/r^2_k}(0) =: U_k.$$ 
Notice that $\varrho_k/r^2_k\to \infty$ and thus following the argument above, there exists some minimal and embedded $\Sigma_2 \subset (\R^{n+1},g_0)$ with finite total curvature such that 
\begin{align*}
\lim_{R\to \infty} \lim_{k\to \infty} \int_{M_k\cap B_{Rr^2_k}(p^2_k)} |A_k|^n \id V_{M_k} &= \lim_{R\to \infty} \lim_{k\to \infty} \int_{\tilde{M}^2_k\cap B_{R}(0)} |\tilde{A}^2_k|^n \id V_{\tilde{M}^2_k}\\
&=  \lim_{R\to \infty} \int_{\Sigma_2 \cap B_R(0)} |A|^n  \id V_{\Sigma_2}\\
&= \mathcal{A}(\Sigma_2).
\end{align*}
Again by Corollary \ref{nontriv} we know that $\Sigma_2$ is non-trivial and connected when $n=2$, and when $n\geq 3$ we know that it is possibly a collection of trivial and non-trivial minimal hypersurfaces, at least one of which is non-trivial. 

\textbf{Case 2: The second fraction in \eqref{step1} is bounded by some $C_0 <\infty$ -- the second bubble forms near the first.} In this case, the first fraction in \eqref{step1} tends to infinity. This time we blow up, centred at the first point, but now with the scale $(C_0+1)r^2_k$, i.e. we let
\begin{equation*}
\tilde{M}^2_k : = \frac{1}{(C_0+1)r^2_k}\, (M_k\cap B_\delta) \subset \Big(B^{\R^{n+1}}_{\delta/(C_0+1)r^2_k}(0), \tilde{g}_k\Big)
\end{equation*}
where $\tilde{g}_k = g_0 + (r_k^2)^2 O(|x|^2)$. By the choice of $(p^2_k,r^2_k)$ we have the existence of some $R<\infty$ so that $\tilde{M}^2_k$ is stable inside every ball of radius $\frac{1}{2(C_0+1)}$ in the annulus
\begin{equation*}
B^{\R^{n+1}}_{\delta/(C_0+1)r^2_k}(0)\setminus B^{\R^{n+1}}_{Rr^1_k/r^2_k}(0)
\end{equation*}
and has finite index. Again, we have 
\begin{equation}\label{mon}
\frac{\h^n(\tilde{M}^2_k \cap B_R)}{R^n} =  \frac{\h^n(M_k\cap B_{(C_0+1)Rr^2_k}(p^1_k))}{((C_0+1)Rr^2_k)^n} \leq C
\end{equation}
for some uniform constant $C$. By assumption we know that $r^1_k/r^2_k\to 0$. By the local compactness theory from Theorem \ref{thm:oldmain} we know that $\tilde{M}_k$ converges locally and smoothly to some limit $\Sigma_2$ away from the only possible bad point at the origin\footnote{Remember that a point is ``bad'' implies that the index coalesces at that point, thus by the choice of $r^2_k$ and $p^2_k$ this cannot happen anywhere except the origin.}. As $\Sigma_1$ is non-trivial, it must have more than one end, thus forcing the origin to be a genuine bad point of convergence on the $r^2_k$ scale. Furthermore, since $B_{r^2_k + (r^2_k)^2}(p^2_k)\in \mathfrak{C}^{\ell}_k$ for some $\ell$ then for some subsequence we have (for all $k$) either  $B_{2(r^2_k + (r^2_k)^2)}(p^2_k)\cap U^{\ell-1}_k = \emptyset$ or $B_{2(r^2_k + (r^2_k)^2)}(p^2_k)\cap U^{\ell-1}_k \neq \emptyset$ and there is some component $C_k$ of $M_k \cap B_{6(r^2_k + (r^2_k)^2)}(p^2_k)$ with $C_k \cap B_{r^2_k + (r^2_k)^2}(p^2_k)$ unstable.

In the former case, we must have that some component of $\tilde{M}_k^2$ converges smoothly and with multiplicity one to a component of $\Sigma_2$ which must be non-trivial. If not, then (letting $\tilde{p}^2_k$ denote $p^2_k$ at this scale) $\tilde{M}_k^2 \cap B_{2}(\tilde{p}^2_k)$ converges locally smoothy to planes (since $\tilde{p}^2_k$ must in this case be far from the origin and the convergence is smooth away from there). But one can use the proof of Corollary \ref{nontriv} to rule this out. 

In the latter we must have that $\tilde{C}_k$ converges smoothly and with multiplicity one to some non-trivial limit and again that $\Sigma_2$ contains some non-trivial component. To see this we observe that $\tilde{C}_k\cap B_{1/2}(p)$ must be stable for all $p\in \tilde{C}_k$. If not, then we have found a smaller ball in the admissible set which gives a contradiction. Then we may apply the proof of Corollary \ref{nontriv} to the component $\tilde{C}_k$. 

In either case the non-trivial component must be a bounded distance away from the origin at this scale, so by Corollary \ref{nontriv}, we must have $\lambda_1(B_{2r^2_k}(p^2_k)\setminus B_{Rr^1_k}(p^1_k))\to -\infty$ for any $R$.

When $n=2$, such other non-trivial components cannot exist by the half-space theorem \cite{HM90}, so the whole of $\Sigma_2$ must be trivial, a contradiction. This shows that it is impossible for Case 2 to occur in dimension $n=2$. 

In either case we find the following.
\begin{equation*}
\lim_{\delta_1\to 0}\lim_{R\to \infty}\lim_{k\to \infty} \int_{M_k\cap B_{Rr^1_k}} |A_k|^n \id V_{M_k}+ \int_{M_k\cap B_{Rr^2_k}\setminus B_{\delta_1 r^2_k}} |A_k|^n \id V_{M_k} = \mathcal{A}(\Sigma_1) + \mathcal{A}(\Sigma_2).  
\end{equation*}

Note that we have not yet controlled the term
\begin{equation*}
\lim_{\delta_1\to 0}\lim_{R\to \infty}\lim_{k\to \infty}\int_{M_k\cap B_{\delta_1 r^2_k}\setminus B_{Rr^1_k}} |A_k|^n \id V_{M_k}.
\end{equation*}
However, we will see that $M_k$ is a neck of order $(\eta (R,\delta_1), L(R,\delta_1))$ here, for some $\eta>0$ and $L\in \N$ as defined in Definition \ref{neckdef}. We will deal with neck regions like this in the following section, where we will see that the above limit actually vanishes.

In Case 2 above, we say that the bubbles $\Sigma_1$ and $\Sigma_2$ form a string (they are forming on each-other) centred at the point $p^1_k$ corresponding to the smallest scale.  

\textbf{Further bubbles.} We continue along our bubble sequence in the same fashion. Either $(r^3_k, p^3_k)$ is assigned to one of the previous bubbles (as in Case 2 above), or it splits off on its own (as in Case 1). Explicitly, we know that 
\begin{equation}\label{step2}
\lim_{k\to \infty} \bigg(\frac{r^3_k}{r^i_k} + \frac{dist_g(p^3_k, p^i_k)}{r^3_k} \bigg) = \infty
\end{equation}
for $i=1,2$. If we were in Case 1 above, then we now have two new subcases.

\textbf{Case 1.1: The second fraction in \eqref{step2} is infinite for both $i=1,2$ -- the bubbles are forming separately.} In this case, again the reader can check that we form a new bubble by following Case 1 above (except we blow up centred at $p^3_k$ at scale $r^3_k$). In particular there exists some $\Sigma_3\subset (\R^{n+1}, g_0)$ with finite total curvature and 
\begin{align*}
\lim_{R\to \infty} \lim_{k\to \infty} \int_{M_k\cap B_{Rr^3_k}(p^3_k)} |A_k|^n \id V_{M_k} &=\lim_{R\to \infty} \lim_{k\to \infty} \int_{\tilde{M}^3_k\cap B_{R}(0)} |\tilde{A}^3_k|^n \id V_{\tilde{M}^3_k}\\
&=  \lim_{R\to \infty} \int_{\Sigma_3 \cap B_R(0)} |A|^n  \id V_{\Sigma_3}\\
&= \mathcal{A}(\Sigma_3).
\end{align*}

Thus, we have 
$$\sum_{i=1}^3 \lim_{R\to \infty} \lim_{k\to \infty} \int_{M_k\cap B_{Rr^i_k}(p^i_k)} |A_k|^n \id V_{M_k} = \sum_{i=1}^3 \mathcal{A}(\Sigma_i).$$

\textbf{Case 1.2: The second fraction in \eqref{step2} is bounded by $C_0<\infty$ for one of $i=1,2$ -- the third bubble is forming on $\Sigma_i$.} In this case, we follow Case 2 above; i.e. we blow up about the first point in this string (in this case $p^i_k$) at a scale $r^3_k$ by letting
$$\tilde{M}^3_k : = \frac{1}{(C_0+1)r^3_k} \, M_k \subset \Big(B^{\R^{n+1}}_{\delta/(C_0+1)r^3_k}(0), \tilde{g}_k\Big)$$ where $\tilde{g}_k = g_0 + (r_k^3)^2 O(|x|^2)$. 
Again, we know that $\tilde{M}^3_k$ is stable inside every ball of radius $\frac{1}{2(C_0+1)}$ in the annulus
$$B^{\R^{n+1}}_{\delta/(C_0+1)r^3_k}(0)\setminus B^{\R^{n+1}}_{Rr^i_k/r^3_k}(0),$$ 
for some $R$ and it has finite index. Moreover, we again have 
\begin{equation}\label{mon}
\frac{\h^n(\tilde{M}^3_k \cap B_R)}{R^n} =  \frac{\h^n(M_k\cap B_{R(C_0+1)r^3_k}(p^1_k))}{((C_0+1)Rr^2_k)^n} \leq C
\end{equation}
for some uniform constant $C$. By assumption, we know that $r^i_k/r^3_k\to 0$. By the compactness result from Corollary \ref{cor:Rn}, we know that $\tilde{M}_k$ converges locally and smoothly to some limit $\Sigma_3$ away from the only possible bad point at the origin. By an argument as in Case 2 above, we get first that this cannot happen when $n=2$, and the existence of some non-trivial $\Sigma_3$ with $\lambda_1(B_{2r^3_k}(p^3_k)\setminus (B_{Rr^2_k}(p^i_k))\to -\infty$ by Corollary \ref{nontriv} and the following 
\begin{equation*}
\lim_{\delta_2\to 0}\lim_{R\to \infty}\lim_{k\to \infty} \int_{M_k\cap B_{Rr^i_k}} |A_k|^n \id V_{M_k} + \int_{M_k\cap B_{Rr^3_k}\setminus B_{\delta_2 r^i_k}} |A_k|^n \id V_{M_k} = \mathcal{A}(\Sigma_i) + \mathcal{A}(\Sigma_3).  
\end{equation*}
Therefore 
\begin{multline*}
\lim_{\delta_2\to 0}\lim_{R\to \infty}\lim_{k\to \infty}\bigg( \int_{M_k\cap B_{Rr^i_k}} |A_k|^n \id V_{M_k} + \int_{M_k\cap B_{Rr^3_k}\setminus B_{\delta_2 r^i_k}} |A_k|^n \id V_{M_k}\\
+\int_{M_k\cap B_{Rr^j_k}(p^j_k)} |A_k|^n \id V_{M_k} \bigg) = \sum_{i=1}^3 \mathcal{A}(\Sigma_i),
\end{multline*}
where $i\neq j$ on the left hand side. Of course, we have again not yet dealt with 
$$\lim_{\delta_2\to 0}\lim_{R\to \infty}\lim_{k\to \infty}\int_{M_k\cap B_{\delta_2 r^3_k}\setminus B_{Rr^i_k}} |A_k|^n \id V_{M_k},$$
but as before, we will see that $M_k \cap \big(B_{\delta_2 r^3_k}\setminus B_{Rr^i_k}\big)$ is a neck of order $(\eta, L)$ for some $\eta>0$ and $L\in\N$. Such regions are investigated in the next section.
 
If instead we were in Case 2 above, then we also have two new subcases.

\textbf{Case 2.1: The second fraction in \eqref{step2} is infinite for both $i=1,2$ -- the third bubble forms separately from the first two.} In this case, follow Case 1.1 above. By blowing up at scale $r^3_k$, we find again that
\begin{align*}
\lim_{R\to \infty} \lim_{k\to \infty} \int_{M_k\cap B_{Rr^3_k}(p^3_k)} |A_k|^n \id V_{M_k}  &=\lim_{R\to \infty} \lim_{k\to \infty} \int_{\tilde{M}^3_k\cap B_{R}(0)} |\tilde{A}^3_k|^n \id V_{\tilde{M}^3_k} \\
& =  \lim_{R\to \infty} \int_{\Sigma_3 \cap B_R(0)} |A|^n  \id V_{\Sigma_3}\\
&= \mathcal{A}(\Sigma_3). 
\end{align*}
Thus 
\begin{multline*}
\lim_{\delta_1\to 0}\lim_{R\to \infty}\lim_{k\to \infty}\bigg( \int_{M_k\cap B_{Rr^3_k}} |A_k|^n \id V_{M_k} + \int_{M_k \cap B_{Rr^2_k}\setminus B_{\delta_1 r^1_k}} |A_k|^n \id V_{M_k} \\
+\int_{M_k\cap B_{Rr^1_k}(p^1_k)} |A_k|^n \id V_{M_k} \bigg) = \sum_{i=1}^3 \mathcal{A}(\Sigma_i),
\end{multline*} 
and we will still have to deal with 
$$\lim_{\delta_1\to 0}\lim_{R\to \infty}\lim_{k\to \infty}\int_{M_k \cap B_{\delta_1 r^2_k}\setminus B_{Rr^1_k}} |A_k|^n \id V_{M_k},$$
which as always is a neck region which will be dealt with in the next section.

\textbf{Case 2.2: The second fraction in \eqref{step2} is finite for one of $i=1,2$ -- the third bubble forms on one of the previous ones.} In this case, we must actually have that the second term in \eqref{step2} is finite for \emph{both} $i=1,2$: If it is finite for $i=1$ then for $i=2$ we have (remember that $r^3_k \geq r^2_k$ by assumption)
$$\frac{dist_g(p^3_k, p^2_k)}{r^3_k} \leq \frac{dist_g(p^3_k, p^1_k)}{r^3_k} + \frac{dist_g(p^1_k, p^2_k)}{r^3_k} \leq \frac{dist_g(p^3_k, p^1_k)}{r^3_k} + \frac{dist_g(p^1_k, p^2_k)}{r^2_k}\leq C <\infty.$$
Also, if it is finite for $i=2$ then when $i=1$ we have
$$\frac{dist_g(p^3_k, p^1_k)}{r^3_k} \leq \frac{dist_g(p^3_k, p^2_k)}{r^3_k} + \frac{dist_g(p^1_k, p^2_k)}{r^3_k} \leq \frac{dist_g(p^3_k, p^2_k)}{r^3_k} + \frac{dist_g(p^1_k, p^2_k)}{r^2_k}\leq C <\infty.$$
Note that now we must have both $r^3_k/r^2_k\to \infty$ and $r^3_k/r^1_k\to \infty$.

Similarly as in Case 2, blow up centred at the first point in the string (in this case $p^1_k$) at a scale $r^3_k$. We leave it as an exercise to check that this is impossible when $n=2$, otherwise $\Sigma_3$ is non-trivial with $\lambda_1(B_{2r^3_k}(p^3_k)\setminus B_{Rr^2_k}(p^1_k))\to -\infty$ and
\begin{multline*}
\lim_{\delta_1\to 0 }\lim_{R\to \infty}\lim_{k\to \infty}\bigg( \int_{M_k\cap B_{Rr^1_k}} |A_k|^n \id V_{M_k} + \int_{M_k\cap B_{Rr^2_k}\setminus B_{\delta_1 r^2_k}} |A_k|^n \id V_{M_k}\\
+\int_{M_k\cap B_{Rr^3_k}\setminus B_{\delta_1 r^3_k}} |A_k|^n \id V_{M_k} \bigg) = \sum_{i=1}^3 \mathcal{A}(\Sigma_i).  
\end{multline*} 

Note that this time, we are still to deal with two regions, namely with
$$\lim_{\delta_1\to 0}\lim_{R\to \infty}\lim_{k\to \infty} \bigg(\int_{M_k\cap B_{\delta_1 r^2_k}\setminus B_{Rr^1_k}} |A_k|^n \id V_{M_k} + \int_{M_k\cap B_{\delta_1 r^3_k}\setminus B_{Rr^2_k}} |A_k|^n \id V_{M_k} \bigg).$$
Both regions are neck regions as above, which we will deal with in the following section.

We can continue down this route for finitely many steps, until we have exhausted all point-scale sequences. Each new scale is either a bubble forming on one of the previous strings of bubbles, or it is occurring on its own scale. Each time we are accounting for all the total curvature \emph{except} on the neck regions. 

If we take a distinct point-scale sequence $(q_k,\varrho_k)$ as in the theorem, then if we blow up at this scale and we end up with something non-trivial in the limit, then we must have captured some more coalescing index in an admissible ball, but this cannot happen since we have exhausted all such regions in the process. 

Therefore Part 1 of Theorem \ref{thm:bubbles} is proved.
\end{proof}

%------------------------------------------------------------------
\section{The neck analysis}\label{NA}
In this section we finish the proof of our main result, Theorem \ref{thm:bubbles} Part 2 (and therefore also Theorem \ref{thm:main}), by showing that the neck regions which we have not yet dealt with in our bubbling argument do not contribute to the total curvature. More precisely, we prove the following lemma.
\begin{lemma}\label{necklemma}
Each neck region encountered in the above bubbling argument is of the form $M_k \cap \big(B_{\delta r^i_k}(p^j_k)\setminus B_{Rr^j_k}(p^j_k)\big)$, with $r^i_k/r^j_k \to \infty$, and we have
\begin{equation}\label{eq:neck}
\lim_{\delta\to 0}\lim_{R\to \infty}\lim_{k\to \infty}\int_{M_k \cap B_{\delta r^i_k}(p^j_k)\setminus B_{Rr^j_k}(p^j_k)} |A_k|^n \id V_{M_k} = 0.
\end{equation}
\end{lemma}
\begin{proof}
By an induction argument, we can assume that on such regions we know that  
$$\lim_{R\to \infty}\lim_{k\to \infty}\int_{M_k \cap Rr^j_k(p^j_k)} |A_k|^n \id V_{M_k} = \sum_\ell \mathcal{A}(\Sigma_\ell) \neq 0,$$
or in other words that all of the total curvature is accounted for on the inner region. In particular this extra hypothesis is true for the first neck region in a string, and once we show \eqref{eq:neck} on this first neck, then this hypothesis will be true for the next neck region and so on.

By rescaling, we can consider (letting $\tilde{M}_k = \tilde{M}_k^i = \frac{1}{r^i_k}M_k$ and $s_k = r^j_k/r^i_k \to 0$) 
$$\lim_{\delta \to 0}\lim_{R\to \infty}\lim_{k\to \infty}\int_{\tilde{M}_k \cap B_{\delta}\setminus B_{Rs_k}} |\tilde{A}_k|^n \id V_{\tilde{M}_k}.$$
We then first claim the following.

\textbf{Claim 1:} $\tilde{M}_k\cap B_{\delta}$ converges as a varifold to a plane as $k\to \infty$. Moreover $\tilde{M}_k$ is a neck of order $(\eta, L)$ as in Definition \ref{neckdef} with
$$\lim_{\delta\to 0}\lim_{R\to \infty}\lim_{k\to \infty} \eta \to 0.$$ 
Thus, after a rotation, $\tilde{M}_k\cap (B_{\delta_1 }\setminus B_{Rs_k})$ is a graph over the plane $\{x^{n+1}=0\}$ with a slope $\eta$ which can be made arbitrarily small for $\delta$ small enough and $R$, $k$ large enough. 

\begin{proof}[Proof of Claim 1] 
First of all, when $n=2$, the whole of $\tilde{M}_k$ converges to a plane, and when $n\geq 3$ we can choose $\delta$ small enough so that the part of $\tilde{M}_k$ inside $B_{\delta}$ converges to a plane. These statements follow since, when $n=2$, we cannot have a bubble forming on another bubble, so we must have bad convergence at the origin at this scale and therefore we converge to a plane passing through the origin by Corollary \ref{cor:Rn}. When $n\geq 3$ on the other hand, we know that $\tilde{M}_k$ converges to a collection of planes and non-trivial minimal hypersurfaces, hence all of the non-trivial hypersurfaces at this scale are therefore some bounded distance away from the origin -- else they are trivial. Thus any part of $\tilde{M}_k$ which is close enough to the origin must converge to a plane passing through the origin. 

Now rotate so that this limit plane coincides with $\{x^{n+1} = 0\}$. For a contradiction  suppose that $\eta \not\to 0$, in which case there exists some sequence $t_k \to 0$ such that $t_k/s_k \to \infty$ and 
$$\lim_{k\to \infty} \int_{\tilde{M}_k \cap(B_{2t_k}\setminus B_{t_k})} |\tilde{A}_k|^n \id V_{\tilde{M}_k}\neq 0.$$
The reader can check that we can blow up again at the scale $t_k$, i.e. we set $\tilde{\tilde{M}}_k = \frac{1}{t_k} \tilde{M}_k$. We must have a point of bad convergence at the origin, thus $\tilde{\tilde{M}}_k$ converges locally smoothly and graphically to a plane away from the origin -- we will show below that this plane must coincide with $\{x^{n+1} =0 \}$. In particular 
$$\lim_{k\to \infty} \int_{\tilde{M}_k \cap(B_{2t_k}\setminus B_{t_k})} |\tilde{A}_k|^n \id V_{\tilde{M}_k} = \lim_{k\to \infty} \int_{\tilde{\tilde{M}}_k\cap(B_{2}\setminus B_{1})} |\tilde{\tilde{A}}_k|^n \id V_{\tilde{\tilde{M}}_k}= 0.$$
Now we assume that the limit plane is independent of $\tilde{t}_k$ and always equals $\{x^{n+1}=0\}$. Notice that over $B_2\setminus B_1$, $\tilde{\tilde{M}}_k$ is uniformly a graph over the limit plane with slope converging to zero (for $k$ sufficiently large, i.e. when $\delta$ is small enough and $R$, $k$ are large enough). Thus the same argument implies that we are graphical with uniformly bounded slope over the whole neck -- and even that the slope converges to zero in this limit. In order to finish the proof of Claim 1 we check that the limit planes in this blow-up are always equal to $\{x^{n+1}=0\}$.

To see this we use a foliation argument similar to that in \cite[claim 6]{Sha15} which uses the ideas in \cite[pp 253 -- 256]{W87}. In particular we know that for all fixed $\tau >0$, on $B_{\delta_1}\setminus B_{\tau}$ each component of $\tilde{M}_k$ converges smoothly and graphically to $\{x^{n+1}=0\}$. Without loss of generality consider one of these graphs, $u_k$ and in particular the component $\mathcal{C}_k$ connected to this graph on $B_{\delta_1}\setminus B_{Rs_k}$. Fix $\tau<\delta_1/2$ and consider the cylinder $C_{2\tau}$ over $B_{2\tau}\cap \{x^{n+1}=0\}$. As in \cite[claim 6]{Sha15} we can foliate an open neighbourhood of $C_{2\tau}\cap\{x^{n+1}=0\}$ by minimal graphs $v_k^h$ such that we have
$$\text{$v_k^h = u_k + h$ on $\partial(C_{2\tau}\cap \{x^{n+1}=0\})$.}$$
Also, 
$$\text{$\|v_k^h-h\|_{C^{2,\alpha}}\leq K \|u_k\|_{C^{2,\alpha}(\partial(C_{2\tau}\cap \{x^{n+1}=0\}))}\leq K\varepsilon_k \to 0$ as $k\to \infty$.}$$  
Similarly as is \cite[Lemma 3.1]{W15} we can define a diffeomorphism of this cylindrical region (via its inverse)
$$F^{-1}_k(x^1,\dots,x^n, y) = (x^1,\dots, x^n, v_k^{y+h_k}(x^1,\dots, x^n))$$
where $h_k \to 0$ is uniquely chosen so that $v_k^{h_k}(0,\dots ,0) = 0.$ Notice that $F_k \to Id$ as $k\to \infty$ in $C^2$. 
 
We now work with these new coordinates $(x^1,\dots, x^n, y)$, on which horizontal slices ($\{y = c\}$) provide a minimal foliation, and furthermore in these coordinates, the part of $\tilde{M}_k$ described by $u_k$ takes a constant value $-h_k$ at the boundary of $C_{2\tau}$. Without loss of generality (by perhaps choosing a sub-sequence) we assume that $-h_k \geq 0$ for all $k$ (if $-h_k \leq 0$ the proof is similar). Now, suppose to the contrary that there exists a sequence $t_k$ whose limit plane is different from $\{x^{n+1} = 0\}$. In this case, and in these coordinates, for all $k$ sufficiently large there exists part of $\mathcal{C}_k$ in the lower half space $\{y<0\}$. This clearly violates the maximum principle -- there exists a slice $\{y=c\}$ such that $\mathcal{C}_k$ touches this slice and locally lies on one side of it. 

The proof of Claim 1 is complete. 
\end{proof}

When continuing with the proof of Lemma \ref{necklemma}, we note that we need to prove \eqref{eq:neck} only for a single sheet of the graph (of which there are finitely many). So let $z^i = x^i|_{x^{n+1} = 0}$ be coordinates on $\{x^{n+1}=0\}$ and pick the function $u_k:B_{\delta r^i_k}\setminus B_{R r^j_k}\cap \{x^{n+1}=0\}\to (\R^{n+1},g_k)$ describing this sheet, noting that we have $g_k = g_0 + (r^i_k)^2 O(|x|^2)$ and $|\D u_k|\leq \eta$. Letting $\hat{u}_k(z_1,\dots, z_n) = (z^1, \dots, z^n, u_k(z^1,\dots, z^n))$, we define $\hat{g}_k = \hat{u}_k^\ast g_k.$ The reader can check easily that there exists some uniform $C$ such that 
\begin{equation*}
|g_k(\hat{u}_k)_{ij} - \delta_{ij}| + |g_k(\hat{u}_k)^{ij} - \delta^{ij}|\leq C\cdot O(|\hat{u}_k|^2)
\end{equation*}
as well as 
\begin{equation*}
|(\Gamma_k)^i_{jl}(\hat{u}_k)|\leq C\cdot O( |\hat{u}_k|),
\end{equation*}
where $(\Gamma_k)^i_{jl}$ are the Christoffel symbols of $g_k$. From this, a short computation shows that 
\begin{equation}\label{eq:ghat1}
|\hat{g}_k(z)_{\alpha\beta} - \delta_{\alpha\beta}| + |\hat{g}_k(z)^{\alpha\beta} - \delta^{\alpha\beta}| \leq C\cdot\big((1+\eta)O(|\hat{u}_k|^2) + \eta^2\big)
\end{equation}
and
\begin{equation}\label{eq:ghat2}
|(\hat{\Gamma}_k)^\gamma_{\alpha\beta}(z)|\leq C\cdot \big((1+\eta)O(|\hat{u}_k|) + (O(|\hat{u}_k|^2) + \eta) |\D^2 u_k|\big),
\end{equation}
where $(\hat{\Gamma}_k)^\gamma_{\alpha\beta}$ are the Christoffel symbols of $\hat{g}_k$. Here, $|\D^2 u_k|$ is just the Euclidean size of the Euclidean Hessian of $u_k$. Finally, we also have $\ed V_{\hat{g}_k} = \sqrt{det(\hat{g}_k)} \ed z^1 \dots \ed z^n$ with
\begin{equation*}
\frac{1}{C}\leq G_k = \sqrt{det(\hat{g}_k)}\leq C.
\end{equation*}
In coordinates, the second fundamental form is given by
$$(A_k)^i_{\alpha\beta}(z) = \ppl{\hat{u}_k^i}{x^\alpha}{x^\beta} - (\hat{\Gamma}_k)^{\gamma}_{\alpha\beta}\pl{\hat{u}^i}{x^\gamma}$$ and thus using the above, we obtain
\begin{equation}\label{eq:AHess}
|A_k|^n_{g_k, \hat{g}_k} \leq C\cdot \big(|\D^2 u_k|^n + \eta^n O(|\hat{u}_k|^n)\big)
\end{equation}
and 
\begin{equation}\label{eq:HessA}
|\D^2 u_k|^n \leq C\cdot\big(|A_k|^n_{g_k, \hat{g}_k} + \eta^n O(|\hat{u}_k|^n)\big)
\end{equation}
Since we now know that
$$\hat{u}_k : (B_{\delta r^i_k}\setminus B_{Rr^j_k}\cap \{x^{n+1}=0\}, \hat{g}_k)\to (B_{\delta r^i_k}\setminus B_{R r^j_k}, g_k)$$ 
is a minimal isometric immersion, the tension field vanishes, and in particular 
\begin{equation}\label{eq:Laplace}
|\Delta_{\hat{g}_k} u_k| =  \Big| \hat{g}_k^{\alpha\beta} \Gamma_k(\hat{u}_k)^{n+1}_{jl} \pl{\hat{u}_k^j}{x^\alpha}\pl{\hat{u}_k^l}{x^\beta} \Big| \leq C\eta^2 |\hat{u}_k|.
\end{equation}
We also introduce the notation
$$f_k : = \hat{g}_k^{\alpha\beta} \Gamma_k(\hat{u}_k)^{n+1}_{jl} \pl{\hat{u}_k^j}{x^\alpha}\pl{\hat{u}_k^l}{x^\beta}.$$

Now, given any $\phi_k\in W_0^{2,n}\big(B_{\delta r^i_k}\setminus B_{Rr^j_k}\cap \{x^{n+1}=0\}\big)$, a standard Calderon-Zygmund estimate (see e.g. \cite[Corollary 9.10]{gt}) tells us that there exists some uniform $C$ with 
\begin{equation}\label{CZ}
\|\D^2(\phi_k u_k)\|_{L^n} \leq C\|\Delta(\phi_k u_k)\|_{L^n}.
\end{equation}
A short computation shows that
\begin{align*}
\Delta(\phi_k u_k) &= \Delta(\phi_k u_k) - \Delta_{\hat{g}_k} (\phi_k u_k) +  \Delta_{\hat{g}_k} (\phi_k u_k)\\
&= \delta^{\alpha\beta}\ppl{(\phi_k u_k)}{x^\alpha}{x^\beta} - \frac{1}{G_k}\pl{}{x^\alpha}\left(\hat{g}_k^{\alpha\beta}G_k \pl{(\phi_k u_k)}{x^\beta}\right) +  \Delta_{\hat{g}_k} (\phi_k ) u_k \\
&\qquad + \phi_k  \Delta_{\hat{g}_k} (u_k) + 2\langle \ed \phi_k, \ed u_k \rangle_{\hat{g}_k}\\
&= (\delta^{\alpha\beta}-\hat{g}_k^{\alpha\beta})\ppl{(\phi_k u_k)}{x^\alpha}{x^\beta} + \hat{g}_k^{\alpha\beta}(\hat{\Gamma}_k)^\gamma_{\alpha\beta} \pl{u_k}{x^\gamma} \phi_k +\hat{g}_k^{\alpha\beta}\ppl{\phi_k}{x^\alpha}{x^\beta} u_k\\
&\qquad + \phi_k f_k + 2\langle \ed \phi_k, \ed u_k \rangle_{\hat{g}_k}.
\end{align*}
Thus, using $|\D u_k|\leq \eta$ and the control we have on the other terms from \eqref{eq:ghat1}, \eqref{eq:ghat2}, \eqref{eq:Laplace}, and assuming also that $\phi_k\leq 1$, we end up with 
\begin{align*}
|\Delta(\phi_k u_k)|^n &\leq C\big(\eta^{2n} + \|\hat{u}_k\|_{L^{\infty}}^{2n}\big) |\D^2(\phi_k u_k)|^n +C\eta^n \|\hat{u}_k\|_{L^{\infty}}^n\\
&\qquad + C\eta^n\big(\eta^n + \|\hat{u}_k\|_{L^{\infty}}^n\big) |\D^2 u_k|^n + C|\D^2 \phi_k|^n |u_k|^n\\
&\qquad+ C\eta^{2n} \|\hat{u}_k\|_{L^{\infty}}+ C\eta^n |\D \phi_k|^n.
\end{align*}
Combining this with \eqref{CZ}, we get (for $\delta$ sufficiently small and $R$, $k$ sufficiently large) 
\begin{equation}\label{3}
\begin{aligned}
\|\D^2(\phi_k u_k)\|_{L^n}^n &\leq C\|\Delta(\phi_k u_k)\|_{L^n}^n \\
&\leq \frac12 \|\D^2(\phi_k u_k)\|^n_{L^n} + C\eta^n\delta^n (r^i_k)^n \|\hat{u}_k\|_{L^{\infty}}^n\\
&\qquad+ C\eta^n(\eta^n + \|\hat{u}_k\|_{L^{\infty}}^n)\int_{\R^n} |\D^2 u_k|^n \id V_{g_0}\\
&\qquad+ C \int_{\R^n} |\D^2 \phi_k|^n|u_k|^n \id V_{g_0} +C\eta^n \int_{\R^n} |\D \phi_k|^n \id V_{g_0},
\end{aligned}
\end{equation}
where $\R^n$ denotes $\{x^{n+1}=0\}$. Therefore, using absorption and \eqref{eq:HessA}, we obtain
\begin{equation}\label{4}
\begin{aligned}
\|\D^2(\phi_k u_k)\|_{L^n}^n &\leq C\eta^n(\eta^n + \|\hat{u}_k\|_{L^{\infty}}^n)\int_{\R^n} |A_k|^n \id V_{\hat{g}_k} + C \int_{\R^n} |\D^2 \phi_k|^n|u_k|^n \id V_{g_0}\\
&\qquad +C\eta^n \int_{\R^n} |\D \phi_k|^n \id V_{g_0} + C\eta^n\delta^n (r^i_k)^n \|\hat{u}_k\|_{L^{\infty}}^n.
\end{aligned}
\end{equation}

We may always assume that $4R r^j_k < \delta r^i_k$ for any fixed $\delta$, $R$ by picking $k$ sufficiently large. We will also assume that $\phi_k\equiv 1$ on $B_{\frac12\delta r^i_k}\setminus B_{2Rr^j_k}\cap \{x^{n+1}=0\}$ so that, using \eqref{eq:AHess}, we get
\begin{equation}\label{1a}
\begin{aligned}
\int_{B_{\delta r^i_k}\setminus B_{Rr^j_k}} |A_k|^n \id V_{\hat{g}_k} &\leq C \int_{B_{\delta r^i_k}\setminus B_{Rr^j_k}} |\D^2 u_k|^n \id V_{g_0} + C\eta^n (\delta r^i_k)^n\\
&\leq C\int_{B_{\frac12\delta r^i_k}\setminus B_{2 Rr^j_k}} |\D^2(\phi_k u_k)|^n \id V_{g_0} \\
&\qquad + C\int_{B_{\delta r^i_k}\setminus B_{\frac12 \delta_1 r^i_k}} |\D^2 u_k|^n \id V_{g_0}\\
&\qquad +C\int_{B_{2 Rr^j_k}\setminus B_{Rr^j_k}} |\D^2 u_k|^n \id V_{g_0}  + C\eta^n (\delta r^i_k)^n.
\end{aligned}
\end{equation}
Using \eqref{4} for the first integral on the right hand side of \eqref{1a} while using \eqref{eq:HessA} for the other two integrals, we then obtain
\begin{equation}\label{1}
\begin{aligned}
\int_{B_{\delta r^i_k}\setminus B_{Rr^j_k}} |A_k|^n \id V_{\hat{g}_k} &\leq C\eta^n\big(\eta^n + \|\hat{u}_k\|_{L^{\infty}}^n\big)\int_{B_{\delta r^i_k}\setminus B_{Rr^j_k}} |A_k|^n \id V_{\hat{g}_k}\\
&\qquad + C \int_{\R^n} |\D^2 \phi_k|^n|u_k|^n \id V_{g_0}+ C\eta^n \int_{\R^n} |\D \phi_k|^n \id V_{g_0}\\
&\qquad + C\int_{B_{\delta_1 r^i_k}\setminus B_{\frac12\delta_1 r^i_k}} |A_k|^n \id V_{\hat{g}_k} +C\int_{B_{2 Rr^j_k}\setminus B_{Rr^j_k}} |A_k|^n \id V_{\hat{g}_k}\\
&\qquad + C\eta^n\delta^n (r^i_k)^n(1+ \|\hat{u}_k\|_{L^{\infty}}^n).
\end{aligned}
\end{equation}
Once again, when $\delta$ is sufficiently small, $R$ and $k$ sufficiently large we get, by absorbing the first term on the right hand side of \eqref{1} and noting that the total curvature on the dyadic annuli on the third line is bounded by $\eta$, 
\begin{equation}\label{eq:A1}
\begin{aligned}
\int_{B_{\delta_1 r^i_k}\setminus B_{Rr^j_k}} |A_k|^n \id V_{\hat{g}_k} &\leq C \int_{\R^n} |\D^2 \phi_k|^n|u_k|^n \id V_{g_0}\\
&\qquad + C\eta^n \int_{\R^n} |\D \phi_k|^n \id V_{g_0} + C\delta^n + C\eta.
\end{aligned}
\end{equation}
We will now work with a specific $\phi_k$ which we assume to be smooth with compact support, $\phi_k \equiv 1$ on $B_{\frac12\delta r^i_k}\setminus B_{2Rr^j_k}\cap \{x^{n+1}=0\}$ and 
\begin{equation}\label{eq:cutoff1}
|\D \phi_k(z)| \leq \twopartdef{\frac{C}{Rr^j_k}}{|z|\in [Rr^j_k, 2Rr^j_k],}{\frac{C}{\delta r^i_k}}{|z|\in [\frac12 \delta r^i_k,\delta r^i_k],}
\end{equation}
as well as
\begin{equation}\label{eq:cutoff2}
|\D^2 \phi_k|\leq \twopartdef{\frac{C}{R^2(r^j_k)^2}}{|z|\in [Rr^j_k, 2Rr^j_k],}{\frac{C}{\delta^2 (r^i_k)^2}}{|z|\in [\frac12 \delta r^i_k,\delta r^i_k].}
\end{equation}
Moreover, we also have the following claim about $u_k$ on the support of $\D^2\phi_k$.

\textbf{Claim 2:}
When $k$ is sufficiently large then 
\begin{equation*}
|u_k(z)|\leq  \fourpartdef{Cr^j_k\log R}{\text{$n=2$ and $|z|\in [Rr^j_k, 2Rr^j_k]$,}}{Cr^i_k \delta^2}{\text{$n=2$ and $|z|\in [\frac12 \delta r^i_k, \delta r^i_k]$,}}{Cr^j_k}{\text{$n\geq 3$ and $|z|\in [Rr^j_k, 2Rr^j_k]$,}}{C r^i_k\delta^2}{\text{$n\geq 3$ and $|z|\in [\frac12 \delta r^i_k, \delta r^i_k]$.}}
\end{equation*}

\begin{proof}[Proof of Claim 2]
In order to prove this claim, we do two further blow-up arguments. First of all, we consider $\tilde{M}_k^i : = \frac{1}{r^i_k} M_k$ and we notice that on $B_{\delta}$ we must have that $\tilde{M}_k^i$ converges smoothly and graphically to $\{x^{n+1} = 0\}$ away from the origin. In particular, for any $\delta$, we can set $k$ sufficiently large so that $|\tilde{u}^i_k| \leq C\delta^2$ for $z\in B_{\delta}\setminus B_{\frac12 \delta}$ where $\tilde{u}^i_k (z) = \frac{1}{r^i_k} u_k (r^i_k z)$. Thus, we have that $|u_k(z)|\leq Cr^i_k \delta^2$ for $z\in B_{\delta r^i_k}\setminus B_{\frac12 \delta r^i_k}$. 

For the second blow-up, we consider $\tilde{M}^j_k : =\frac{1}{r^j_k} M_k$. Now we know that $\tilde{M}^j_k$ converges locally smoothly and graphically to some limit surfaces $\Sigma_\ell$ away from the origin. If $\Sigma_\ell$ is non-trivial for some $\ell$, the convergence is smooth everywhere and single sheeted, else we are converging with multiplicity to a plane. We point out that by embeddedness we must have that the ends of $\Sigma_\ell$ are all parallel. Furthermore, we now claim that the ends are all parallel to $\{x^{n+1} = 0\}$. This follows since $\tilde{u}^j_k(z):=\frac{1}{r^j_k} u_k (r^j_k z)$ is the graph of some sheet of $\tilde{M}^j_k$. Moreover we know that this is converging smoothly to $\Sigma_\ell \cap B_{2R}\setminus B_R$ for any fixed $R$ as $k\to \infty$. But we know also that 
$$\lim_{\delta\to 0}\lim_{R\to\infty}\lim_{k\to \infty}\|\D \tilde{u}^j_k \|_{L^{\infty}} = 0,$$
so therefore all the ends of $\Sigma_\ell$ are parallel to $\{x^{n+1} = 0\}$. We can now use Schoen's description \cite{Sc83} of the ends of minimal hypersurfaces to conclude that, for $z\in B_{2R}\setminus B_R$ and $k$ sufficiently large:
\begin{equation*}
|\tilde{u}^j_k(z)|\leq \twopartdef{C\log R}{n=2}{C}{n\geq 3.}
\end{equation*}
This finishes the proof of Claim 2.
\end{proof}

Using Claim 2, and a test function $\phi_k$ as described above, i.e. in particular satisfying \eqref{eq:cutoff1} and \eqref{eq:cutoff2}, we obtain from the estimate \eqref{eq:A1}
\begin{equation*}
\int_{M_k\cap B_{\delta r^i_k}\setminus B_{Rr^j_k}} |A_k|^n \id V_{\hat{g}_k} \leq C\frac{(\log R)^n}{R^n} + C\delta^n + C\eta, 
\end{equation*}
and therefore by Claim 1 
$$\lim_{\delta\to 0}\lim_{R\to \infty}\lim_{k\to \infty}\int_{M_k\cap B_{\delta r^i_k}(p^j_k)\setminus B_{Rr^j_k}(p^j_k)} |A_k|^n=0.$$
This finishes the proof of the Lemma. \end{proof}

We can now prove the full version of the bubbling result which will also conclude the result of the main Theorem \ref{thm:main};

\begin{proof}[Proof of Theorem \ref{thm:bubbles} Part 2.]
We can now construct a cover of $B_{\delta}(y)$ in terms of bubble regions, neck regions and empty regions. We call a region \emph{empty} if it is defined by a distinct point-scale sequence $(s_k, q_k)$ (distinct in the sense that all other bubbles occur at different scales, moreover that $B_{s_k}(q_k)$ is disjoint from all the other bubble regions). We can check easily that the total curvature must converge to zero over such a region and we converge smoothly to a collection of planes; if not, when we re-scale we converge to something non-trivial, or we have a point of bad convergence, and have thus found a new region where we have index coalescing -- contradicting the fact that we have exhausted such regions. We leave the reader to do this as it is a standard argument. 

Thus Lemma \ref{necklemma} along with the proof of Theorem \ref{thm:bubbles} Part 1. completes all but the final statement. 
 
This last statement follows since when $k$ is sufficiently large, we have produced a covering whereby all the $M_k$'s are diffeomorphic to one another inside each element of the covering.
\end{proof}

%-------------------------------------------------------------------
\section{Curvature quantisation for $\mathcal{M}_p(\Lambda,\mu)$ and proofs of the corollaries from the introduction}
In this section, we will prove Theorem \ref{cor:lambda} as well as the two Corollaries \ref{cor:Betti} and \ref{cor:equivalence}.

To prove Theorem \ref{cor:lambda}, we will carry out the same bubbling argument as in the proof of Theorem \ref{thm:main} -- in fact, the main approach is the same, except we do not have an index bound and therefore it is not clear whether or not the process stops. Therefore the only thing we check here is that finding regions of coalescing index does indeed stop eventually. Whilst this does not directly prove an index bound, it does indeed prove a qualitative total curvature bound, thus implying an index bound \emph{a-posteriori} by the work of Cheng-Tysk \cite{CT94}. 

\begin{proof}[Proof of Theorem \ref{cor:lambda}]
The compactness theorem in \cite{ACS15} gives us a limiting $M$ whereby a subsequence of the $\{M_k\}$'s converge smoothly and graphically away from a finite set (which we again denote by $\mathcal{Y}$). Thus, we can begin finding point-scale sequences exactly as in Section \ref{sect.bubbling} before. 

\textbf{Claim:} The process of finding point-scale sequences stops after at most $p-1$ iterations. 

To prove the claim, we start with a blow-up argument -- even if we have not found all of the point-scale sequences. Looking back over the proof in Section \ref{sect.bubbling} we see that each time we perform a blow up we construct a new disjoint region where $\lambda_1(M_k \cap B_{2r^i_k}(p^i_k)\setminus U_k)\to -\infty$. Moreover, by Corollary \ref{nontriv} all of the ignored point-scale sequences also contribute a separate region with this property. Thus we can find at most $p-1$ such point-scale sequences and the process stops (see \cite[Lemma 3.1]{ACS15}). This proves the claim.

Due to the above claim, we can now follow precisely the same procedure as in the proof of Theorem \ref{thm:main}, yielding a proof of Theorem \ref{cor:lambda} as well.
\end{proof}

\begin{proof}[Proof of Corollary \ref{cor:Betti}]
All the statements in this corollary are immediate from Theorem \ref{thm:main}, except perhaps that $b_i(M)\leq C$. This last statement could follow from the finiteness of diffeomorphism types result, however one can also use the total curvature bound directly which we describe here. When $n=2$, this is an exercise in applying the Gauss equations followed by the Gauss-Bonnet theorem. 

When $n\geq 3$, we first consider the case that $M$ is an orientable minimal hypersurface at which point the Bochner formula of Savo \cite{S14} states that 
\begin{equation*}
\Delta_p = \D^{\ast}\D - \mathscr{B}^{[p]}.
\end{equation*}
Here, $\Delta_p$ is the $p^{th}$ Hodge-Laplacian, $\D^{\ast} \D$ is the rough Laplacian in the bundle of $p-$forms and $\mathscr{B}^{[p]} = -\mathscr{B}_N^{[p]} + A^{[p]}$ is a self-adjoint operator in this bundle. In particular, when the eigenvalues $\sigma_N$ of the curvature operator $Rm_N$ of $N$ are bounded between $\gamma \leq \sigma_N \leq \Gamma$, we have 
\begin{equation*}
p(n-p)\gamma |\omega|^2\leq \langle\mathscr{B}_N^{[p]}(\omega), \omega\rangle \leq p(n-p)\Gamma|\omega|^2.
\end{equation*}
Furthermore, 
\begin{equation*}
\langle A^{[p]}(\omega), \omega\rangle  \twopartdef{=\frac{1}{2}|A|^2|\omega|^2}{n=2, p=1}{ \leq \min\{p,n-p\}|A|^2|\omega|^2}{n\geq 3.}
\end{equation*}
See \cite{S14} for more details. We now apply Theorem 4 of Cheng-Tysk \cite{CT94} which states that, when $M$ is minimally immersed, we have 
\begin{align*}
b_i(M) &\leq C(N) {n\choose{p}}\int_M \big(\max\{1, \min\{p,n-p\}|A|^2- p(n-p) \gamma\} \big)^{n/2} \id V_M\\
&\leq C(N)(\h^n(M) + \mathcal{A}(M)).
\end{align*}
When $M$ is non-orientable, simply consider the orientable double cover as a minimal immersion and we obtain the same estimate (the Betti numbers of non-orientable manifolds are estimated from above by the Betti numbers of their orientable double covers). 
\end{proof}

\begin{remark}
We point out here a consequence of Savo's Bochner formula above. For simplicity we work now with $n=2$ and $M$ orientable. If the eigenvalues of the curvature operator $Rm_N$ are all bounded from below by some $\gamma >0$ then assuming 
\begin{equation}\label{last}
-\langle \mathscr{B}_N^{[1]}(\omega), \omega \rangle + \langle A^{[1]}(\omega), \omega\rangle < 0 
\end{equation}
implies that $\Delta_1$ is strictly positive definite and thus $M$ is a sphere. Without assuming that $M$ is orientable, by using the orientable double cover, we can also conclude that \eqref{last} implies that $M$ is a sphere or a projective plane. In particular, if 
$$|A|^2 <  2\gamma$$ everywhere then \eqref{last} holds. Thus, any minimal surface with non-positive Euler characteristic in $N^3$, with $\gamma >0$, must have $|A|^2 \geq 2\gamma$ somewhere. In particular we recover \cite[Corollary 5.3.2]{S68} and generalise this result when $n=3$. 

Similar statements can be made in higher dimensions except we have to replace sphere with homology sphere.   
\end{remark}

\begin{proof}[Proof of Corollary \ref{cor:equivalence}]
This follows easily by the fact that we end up with an \emph{a-priori} total curvature bound from Theorem \ref{cor:lambda} and then the result of Cheng-Tysk \cite{CT94} finishes the proof. 
\end{proof}

\makeatletter
\def\@listi{%
  \itemsep=0pt
  \parsep=1pt
  \topsep=1pt}
\makeatother
{\fontsize{10}{11}\selectfont

\vspace{10mm}

Reto Buzano (M\"uller)\\
{\sc School of Mathematical Sciences, Queen Mary University of London, London E1 4NS, United Kingdom}\\

Ben Sharp\\
{\sc Centro di Ricerca Matematica Ennio De Giorgi, Scuola Normale Superiore, 56100 Pisa, Italy}\\


\begin{thebibliography}{99}

\bibitem{ACS15}
L.~Ambrozio, A.~Carlotto and B.~Sharp, \textit{Compactness of the Space of Minimal Hypersurfaces with Bounded Volume and $p$-th Jacobi Eigenvalue}, J. Geom. Anal., DOI 10.1007/s12220-015-9640-4 (2015).

\bibitem{A68}
F.~J.~Almgren, Jr., \textit{Existence and Regularity Almost Everywhere of Solutions to Elliptic Variational Problems among Surfaces of Varying Topological Type and Singularity Structure}, Ann. of Math. \textbf{87}(2) (1968), 321--391.

\bibitem{CT94}
S.~Y.~Cheng and J.~Tysk, \textit{Schr\"odinger operators and index bounds for minimal submanifolds}, Rocky Mountain J. Math. \textbf{24}(3) (1994), 977--996.

\bibitem{CKM15}
O.~Chodosh, D.~Ketover and D.~Maximo, \textit{Minimal surfaces with bounded index}, Preprint (2015), ArXiv:1509.06724.

\bibitem{CM14}
O.~Chodosh and D.~Maximo, \textit{On the topology and index of minimal surfaces}, Preprint (2014), to appear in J. Diff. Geom.

\bibitem{CSc85}
H.~I.~Choi and R.~Schoen, \textit{The space of minimal embeddings of a surface into a three-dimensional manifold of positive Ricci curvature}, Invent. Math. \textbf{81} (1985), 387--394.

\bibitem{EM08}
N.~Ejiri and M.~Micallef, \textit{Comparison between second variation of area and second variation of energy of a minimal surface}, Adv. Calc. Var. \textbf{1}(3) (2008), 223--239.

\bibitem{FC85}
D.~Fischer-Colbrie, \textit{On complete minimal surfaces with finite index in three manifolds}, Invent. Math. \textbf{82} (1985), 121--132.

\bibitem{gt}
D.~Gilbarg and N.~S.~Trudinger, \textit{Elliptic partial differential equations of second order}, Classics in Mathematics. Springer-Verlag, Berlin, 2001.

\bibitem{GNS14}
A.~Grigor'yan, N.~Nadirashvili and Y.~Sire, \textit{A lower bound for the number of negative eigenvalues of Schr\"odinger operators}, Preprint (2014), ArXiv:1406.0317.

\bibitem{HM90}
D.~Hoffman and W.~H.~Meeks III, \textit{The strong halfspace theorem for minimal surfaces}, Invent. Math. \textbf{101}(2) (1990), 373--€"377.

\bibitem{LW02}
P.~Li  and J.~Wang, \textit{Minimal hypersurfaces with finite index}, Math. Res. Lett. \textbf{9}(1) (2002), 95--104.

\bibitem{LY83}
P.~Li and S.~T.~Yau, \textit{On the {S}chr\"odinger equation and the eigenvalue problem}, Comm. Math. Phys. \textbf{88}(3) (1983), 309--318. 

\bibitem{MN14}
F.~C.~Marques and A.~Neves, \textit{Min-Max theory and the Willmore conjecture}, Ann. Math. \textbf{179}(2) (2014), 683--782.

\bibitem{MN13}
F.~C.~Marques and A.~Neves, \textit{Existence of infinitely many minimal hypersurfaces in positive Ricci curvature}, Preprint (2013), ArXiv:1311.6501.

\bibitem{P81}
J.~Pitts, \textit{Existence and regularity of minimal surfaces on {R}iemannian manifolds}, Princeton University Press, Princeton, N.J.; University of Tokyo Press, Tokyo. Mathematical Notes \textbf{27} (1981). 

\bibitem{R95}
A.~Ros, \textit{Compactness of spaces of properly embedded minimal surfaces with finite total curvature}, Indiana Univ. Math. J. \textbf{44}(1) (1995), 139--152.

\bibitem{S14}
A.~Savo, \textit{The {B}ochner formula for isometric immersions}, Pacific J. Math. \textbf{272}(2) (2014), 395--422.

\bibitem{Sc83}
R.~Schoen, \textit{Uniqueness, symmetry, and embeddedness of minimal surfaces}, J. Diff. Geom. \textbf{18} (1983), 791--809.
   
\bibitem{ScS81}
R.~Schoen and L.~Simon, \textit{Regularity of stable minimal hypersurfaces}, Comm. Pure Appl. Math. \textbf{34} (1981), 741--797.

\bibitem{Sha15}
B.~Sharp, \textit{Compactness of minimal hypersurfaces with bounded index}, Preprint (2015) arXiv:1501.02703, to appear in J. Diff. Geom. 

\bibitem{S68}
J.~Simons, \textit{Minimal Varieties in Riemannian Manifolds}, Ann. of Math. \textbf{88} (1968), 62--105.

\bibitem{T89}
J.~Tysk, \textit{Finiteness of index and total scalar curvature for minimal hypersurfaces}, Proc. Amer. Math. Soc. \textbf{105}(2) (1989), 428--435. 

\bibitem{U90}
F.~Urbano, \textit{Minimal surfaces with low index in the three-dimensional sphere},  Proc. Amer. Math. Soc. \textbf{108}(4) (1990), 989--992.

\bibitem{W87}
B.~White, \textit{Curvature estimates and compactness theorems in 3-manifolds for surfaces that are stationary for parametric elliptic functionals}, Invent. Math. \textbf{88}(2) (1987), 243--256.
    
\bibitem{W15}
B.~White, \textit{On the Compactness Theorem for Embedded Minimal Surfaces in 3-manifolds with Locally Bounded Area and Genus}, Preprint (2015), ArXiv:1503.02190.

\end{thebibliography}
\end{document}